\documentclass[12pt]{amsart}
\usepackage[margin=1.35in]{geometry}
\usepackage{amscd,amsmath,amsxtra,amsthm,amssymb,stmaryrd,xr,mathrsfs,mathtools,enumerate,commath, comment}
\usepackage{stmaryrd}
\usepackage{xcolor}
\usepackage{commath}
\usepackage{comment}
\usepackage{tikz-cd}
\usepackage{longtable} 
\usepackage{pdflscape} 
\usepackage{booktabs}
\usepackage{hyperref}
\definecolor{vegasgold}{rgb}{0.77, 0.7, 0.35}
\definecolor{darkgoldenrod}{rgb}{0.72, 0.53, 0.04}
\definecolor{gold(metallic)}{rgb}{0.83, 0.69, 0.22}
\hypersetup{
 colorlinks=true,
 linkcolor=darkgoldenrod,
 filecolor=brown,      
 urlcolor=gold(metallic),
 citecolor=darkgoldenrod,
 }

\usepackage[all,cmtip]{xy}

\DeclareFontFamily{U}{wncy}{}
\DeclareFontShape{U}{wncy}{m}{n}{<->wncyr10}{}
\DeclareSymbolFont{mcy}{U}{wncy}{m}{n}
\DeclareMathSymbol{\Sh}{\mathord}{mcy}{"58}
\usepackage[T2A,T1]{fontenc}
\usepackage[OT2,T1]{fontenc}

\newtheorem{theorem}{Theorem}[section]
\newtheorem{lemma}[theorem]{Lemma}

\newtheorem*{theorem*}{Theorem}
\newtheorem*{ass*}{Assumption}
\newtheorem{definition}[theorem]{Definition}

\newtheorem{remark}[theorem]{Remark}

\newtheorem{proposition}[theorem]{Proposition}

\newcommand{\cF}{\mathcal{F}}

\newcommand{\cK}{\mathcal{K}}

\newcommand{\cG}{\mathcal{G}}

\newcommand{\Z}{\mathbb{Z}}
\newcommand{\p}{\mathfrak{p}}
\newcommand{\Q}{\mathbb{Q}}
\newcommand{\F}{\mathbb{F}}

\newcommand{\cL}{\mathcal{L}}
\newcommand{\cC}{\mathcal{C}}
\newcommand{\cN}{\mathcal{N}}
\newcommand{\cO}{\mathcal{O}}
\newcommand{\cS}{\mathcal{S}}

\newcommand{\op}[1]{\operatorname{#1}}

\newcommand\mtx[4] { \left( {\begin{array}{cc}
 #1 & #2 \\
 #3 & #4 \\
 \end{array} } \right)}

\numberwithin{equation}{section}

\begin{document}

\title[Deformations with large $p$-rank]{Deformations of reducible Galois representations with large Selmer $p$-rank}

\author[E.~Ghate]{Eknath Ghate}
\address[Ghate]{Tata Institute of Fundamental Research, Homi Bhabha
Road, Mumbai - 400005, India}
\email{eghate@math.tifr.res.in}

\author[A.~Ray]{Anwesh Ray}
\address[Ray]{Chennai Mathematical Institute, H1, SIPCOT IT Park, Kelambakkam, Siruseri, Tamil Nadu 603103, India}
\email{anwesh@cmi.ac.in}


\maketitle

\begin{abstract}
Let $p\geq 5$ be a prime number. In this paper, we construct Galois representations associated with modular forms for which the dimension of the $p$-torsion in the Bloch-Kato Selmer group can be made arbitrarily large. Our result extends similar results known for small primes, such as Matsuno's work on Tate-Shafarevich groups of elliptic curves. Extending the technique of Hamblen and Ramakrishna, we lift residually reducible Galois representations to modular representations for which the associated Greenberg Selmer groups are minimally generated by a large number of elements over the Iwasawa algebra. We deduce that there is an isogenous lattice for which the Bloch--Kato Selmer group has large $p$-rank.
\end{abstract}

\section{Introduction}
\subsection{Main result}
\par Let $h$ be a normalized Hecke eigencuspform of weight $k\geq 2$ and $p$ be an odd prime number. Let $F$ be the number field generated by the Fourier coefficients of $h$. Let $\p|p$ be a prime of $F$, set $\cK:=F_{\p}$ and $\cO:=\cO_{\cK}$. Associated to $h$ is Deligne's $2$-dimensional representation \[\rho_{V_h}: \op{Gal}(\bar{\Q}/\Q)\rightarrow \op{Aut}_{\cK}(V_h)\xrightarrow{\sim} \op{GL}_2(\cK).\] Let $V$ be a self-dual twist of $V_h$, i.e., $V$ is isomorphic to its Tate-dual $\op{Hom}(V, \cK(1))$. Choose a Galois stable $\cO$-lattice $T$ contained in $V$, and let
\[\rho_T:\op{Gal}(\bar{\Q}/\Q)\rightarrow \op{Aut}_{\Z_p}(T)\xrightarrow{\sim} \op{GL}_2(\cO)\] be the integral Galois representation. Let $A:=V/T$ and let $\op{Sel}_{\op{BK}}(A/\Q)$ be the associated Bloch-Kato Selmer group. In this article we construct Galois representations for which $\dim_{\cO/\p}\op{Sel}_{\op{BK}}(A/\Q)[\p]$ is arbitrarily large. Related results have been proven in various contexts for small prime numbers. Matsuno \cite{matsuno} proved, using the Iwasawa theory of elliptic curves that for primes $p\in \{2,3,5,7, 13\}$ and any number $N>0$, there exists an elliptic curve $E_{/\Q}$ for which $\dim_{\F_p}\Sh(E/\Q)[p]\geq N$. For the prime $p=2$, see the work by Flynn \cite{Flynn} who proves this result for certain abelian surfaces. Bartel \cite{AlexBartel} studied the growth of the Tate--Shafarevich group in certain dihedral extensions of the form $D_{2p}$. Related work of \v Cesnavi\v cius \cite{cesnavicius} studies the growth of the Selmer group of an abelian variety over a number field $K$ in $p$-cyclic extensions $K'/K$. In particular, the $p$-torsion in the Selmer group over $K'$ can have large dimension. The distribution of Tate--Shafarevich groups in quadratic twist families of abelian varieties has been considered by Bhargava et. al. \cite{BKLS}. Our main result is stated below.  

\begin{theorem}\label{main thm}
   Let $p\geq 5$ be a prime number and $N$ be a natural number. Then, there exists a Galois representation $\rho: \op{G}_{\Q}\rightarrow \op{Aut}_{\Z_p}(T)\xrightarrow{\sim}\op{GL}_2(\Z_p)$ arising from a $p$-ordinary Hecke eigen-cuspform of weight $2$ and a choice of prime $\p|p$, such that 
   \[\dim_{\F_p}\op{Sel}_{\op{BK}}(A/\Q)[p]\geq N.\]
\end{theorem}
In fact, we prove that any reducible Galois representation $\bar{\rho}:\op{G}_{\Q}\rightarrow \op{GL}_2(\F_p)$ with $p\geq 5$ and satisfying some additional conditions lifts to a modular representation $\rho$ for which \[\dim_{\F_p}\op{Sel}_{\op{BK}}(A/\Q)[p]\geq N.\] We refer to Theorem \ref{main thm aux} for further details.
\subsection{Methodology}
\par Matsuno's construction focuses on the relationship between the Selmer groups of a pair of elliptic curves, $E$ and $E'$, over $\mathbb{Q}$ that are related by a rational cyclic $p$-isogeny $\phi: E \rightarrow E'$. These Selmer groups are considered over the cyclotomic $\mathbb{Z}_p$-extension of $\mathbb{Q}$, where their algebraic properties are intertwined. Under specific additional local conditions, Matsuno demonstrates that the $p$-torsion in the Tate-Shafarevich group of $E'$ can become large. For Galois representations $A$ arising from modular forms with ordinary reduction at $p$, similar results can be proven. This is achieved by adapting the methods of Hamblen and Ramakrishna \cite{hamblenramakrishna}, who construct modular lifts of residually reducible Galois representations. The modularity of such representations follows from the results of Skinner and Wiles \cite{skinnerwiles}.

In greater detail, we start with a residual representation \[\bar{\rho}= \begin{pmatrix} \varphi & * \\ 0 & 1 \end{pmatrix}: \operatorname{Gal}(\overline{\mathbb{Q}}/\mathbb{Q}) \rightarrow \operatorname{GL}_2(\mathbb{F}_p),\] which is then lifted to a modular Galois representation

\[
\rho  : \operatorname{Gal}(\overline{\mathbb{Q}}/\mathbb{Q}) \rightarrow \operatorname{GL}_2(\mathbb{Z}_p)
\]
that satisfies certain additional local conditions. Let $T$ be a self-dual twist of the underlying $\mathbb{Z}_p$-module, and let $A := T \otimes_{\mathbb{Z}_p} \mathbb{Q}_p/\mathbb{Z}_p$ be the associated Galois module. Since $\bar{\rho}$ is reducible, there exists another Galois-stable lattice $T'$ containing $T$ as a submodule of index $p$. Defining $A' := T' \otimes_{\mathbb{Z}_p} \mathbb{Q}_p/\mathbb{Z}_p$, we obtain a rational $p$-cyclic isogeny $\phi: A \rightarrow A'$, i.e., $\phi$ is a $\op{G}_\Q$-equivariant surjective homomorphism whose kernel is isomorphic to $\Z/p\Z$.

Let $\Q_\infty$ be the cyclotomic $\Z_p$-extension of $\Q$, set $\Gamma:=\op{Gal}(\Q_\infty/\Q)$ and let $\op{Sel}_{\op{Gr}}(A/\Q_\infty)$ be the Greenberg Selmer group of $A$ over $\Q_\infty$. Let $S$ be the set of primes at which $\rho$ is ramified. There is an exact sequence
\[0\rightarrow \op{Sel}_{\op{BK}}(A/\Q)[p]\rightarrow \op{Sel}_{\op{Gr}}(A/\Q_\infty)^{\Gamma}[p]\rightarrow \prod_{\ell\in S} \op{ker}h_{\ell,A}[p]\] relating the Greenberg Selmer group of $A$ to the Bloch--Kato Selmer group of $A$ and the local Tamagawa numbers $\op{ker}h_{\ell,A}$. There is a similar sequence for $A'$. We show that if the middle term of the above exact sequence is large for $A$, then it is large for $A'$ as well. The local Tamagawa number conditions on $A$ ensure that the middle term is large for $A$. Hence the middle term for $A'$ is also large. Furthermore, the Tamagawa numbers of $A'$ are trivial for most $\ell\in S$. We deduce that $\dim \operatorname{Sel}_{\operatorname{BK}}(A'/\mathbb{Q})[p]$ is large. 

\subsection{Organization} Including the introduction, the article consists of five sections. Section \ref{s 2} is preliminary in nature. We discuss the Iwasawa theory of Bloch--Kato and Greenberg Selmer groups as well as their algebraic properties. We introduce the Iwasawa algebra and the Iwasawa invariants associated to the Greenberg Selmer group. In section \ref{s 3} we discuss lower bounds for the dimensions of the $p$-ranks of the Selmer groups. We study the specific case when there are two lattices $A$ and $A'$ which are related by a $p$-cyclic isogeny, and the relationship between their Bloch--Kato Selmer groups. In section \ref{s 4} we discuss the deformation theory of reducible Galois representations and discuss a variant of the strategy of Hamblen and Ramakrishna. In section \ref{s 5}, we prove Theorem \ref{main thm aux} and the main result, i.e., Theorem \ref{main thm}. 
\subsection{Acknowledgments} The authors thank S.~Hamblen and R.~Ramakrishna for helpful comments. The first author thanks T.~G\'omez and ICMAT, Madrid for hospitality in April 2025. The second author is grateful to TIFR, Mumbai for hosting him in February 2024.

\section{Iwasawa theory of Selmer groups}\label{s 2}
\par In this section, we discuss the Iwasawa theory of Selmer groups associated to a Galois stable lattice that arises from a Hecke eigencuspform. We begin by introducing the Bloch--Kato and Greenberg Selmer groups. Then we introduce Iwasawa theoretic invariants associated to Greenberg Selmer groups considered over the cyclotomic $\Z_p$-extension of $\Q$. These invariants encapsulate the algebraic structure of Selmer groups, viewed as modules over the Iwasawa algebra. For further details, we refer to the original works of Greenberg \cite{Greenbergpadic, Greenbergstructure} and Bloch--Kato \cite{BlochKato}. 

\subsection{The Bloch--Kato and Greenberg Selmer groups} Throughout, we fix an algebraic closure $\bar{\Q}$ of $\Q$ and set $\op{G}_{\Q}:=\op{Gal}(\bar{\Q}/\Q)$. Without further mention, any algebraic extension of $\Q$ considered in this article will be assumed to be contained in the chosen algebraic closure $\bar{\Q}$. Given a prime $\ell$, set $\op{G}_\ell:=\op{Gal}(\bar{\Q}_\ell/\Q_\ell)$, and choose an embedding $\iota_\ell: \bar{\Q}\hookrightarrow \bar{\Q}_\ell$. Note that $\iota_\ell$ induces an inclusion $\op{G}_\ell\hookrightarrow \op{G}_\Q$. Let $\tau$ be the variable in the complex upper half plane and set $q:=\op{exp}\left(2 \pi i \tau\right)$. Consider a normalized Hecke eigencuspform $h=\sum_{n\geq 1} a_n q^n$ on $\Gamma_1(N)$ of weight $k\geq 2$ and let $F_h:=\Q(\{a_n\mid n\geq 1\})$ be the number field generated by the Fourier coefficients of $h$. Fix an odd prime number $p$ and let $\p$ be a choice of prime of $F_h$ that lies above $p$. Let $\cK$ be the completion of $F_h$ at $\p$ and $\cO$ be the valuation ring of $\cK$. Let $\varpi$ be the uniformizer of $\cO$ and $\kappa:=\cO/(\varpi)$ be the residue field of $\cO$. Associated to $h$ is the Galois representation $\rho_{h, \p}: \op{G}_\Q\rightarrow \op{GL}_2(\cK)$, introduced by Eichler, Shimura and Deligne. Set $V_h\simeq \cK^2$ to be the underlying $\cK$-vector space. Let $\chi$ be the cyclotomic character and $\varepsilon$ be the nebentype character. Note that there is a perfect pairing 
\[V_h\times V_h\rightarrow \cK(\chi^{k-1}\varepsilon).\]
We shall assume that there is a character $\psi: \op{G}_{\Q}\rightarrow \cO^\times$ such that $\psi^2=\chi^{2-k}\varepsilon^{-1}$. Consider a self-dual twist $V:=V_h(\psi)$. Our constructions will depend on a choice of Galois stable $\cO$-lattice $T\subset V$. Setting $A:=V/T$, we find that $A$ is $p$-divisible and is isomorphic to $\left(\cK/ \cO\right)^2$. Let $\rho_V$ (resp. $\rho_T$) be the Galois representation associated to $V$ (resp. $T$)
\[\begin{split}
    & \rho_V: \op{G}_{\Q}\rightarrow \op{Aut}_{\cK}(V)\xrightarrow{\sim} \op{GL}_2(\cK);\\
    & \rho_T: \op{G}_{\Q}\rightarrow \op{Aut}_{\cO}(T)\xrightarrow{\sim} \op{GL}_2(\cO).\\
\end{split}\]
Let $\bar{T}$ be the mod $\p$ reduction of $T$ and $\bar{\rho}_T$ be the associated Galois representation
\[\bar{\rho}_{T}: \op{G}_{\Q}\rightarrow \op{Aut}_{\kappa}(\bar{T})\xrightarrow{\sim} \op{GL}_2(\kappa). \]
The representation $\bar{\rho}_T$ is referred to the \emph{residual representation} associated to $\rho_T$. Given any two Galois stable lattices $T$ and $T'$ in $V$, then $\bar{\rho}_T$ is not always isomorphic to $\bar{\rho}_{T'}$. However, if $\bar{\rho}_T$ is absolutely irreducible, then, there is a unique Galois stable lattice $T$ up to scaling. The case that shall interest us in this article is when $\bar{\rho}_T$ is reducible.
\par Setting $\iota_{\p}: F_h\rightarrow \cK$ to denote the natural inclusion, we assume that $h$ is \emph{$\p$-ordinary}, i.e., $\iota_\p(a_p)$ is a unit in $\cO$. Let $\op{I}_{p}$ be the inertia subgroup of $\op{G}_{p}$. There is a natural filtration 
\[0\rightarrow V_h^+\rightarrow V_h\rightarrow V_h^-\rightarrow 0\] of $\cK[\op{G}_{p}]$-modules, where $\op{dim}_{\cK}V_h^{\pm}=1$. Twisting by $\psi$, we obtain the exact sequence 
\[0\rightarrow V^+\xrightarrow{\iota} V\xrightarrow{\pi} V^-\rightarrow 0,\] where $V^{\pm}:=V_h^{\pm}(\psi)$. This induces a short exact sequence
\begin{equation}\label{ses T}0\rightarrow T^+\rightarrow T\rightarrow T^-\rightarrow 0\end{equation} of $\cO[\op{G}_p]$-modules. Here, $T^+:=T\cap V^+$ and $T^-:=\pi(T)$. As is well known, $T^+$ and $T^-$ are free $\cO$-modules of rank $1$. Setting $A^{\pm}:=V^{\pm}/T^{\pm}$, we find that \eqref{ses T} induces the following short exact sequence 
\[0\rightarrow A^+\rightarrow A\rightarrow A^-\rightarrow 0. \]
As $\cO$-modules, $A^+$ and $A^-$ are isomorphic to $\cK/\cO$.
\par Given a natural number $n$, let $\Q(\mu_{p^n})$ denote the cyclotomic extension of $\Q$ generated by the $p^n$-th roots of unity, set $\Q(\mu_{p^\infty}):=\bigcup_{n} \Q(\mu_{p^n})$. Let $\Q_\infty$ denote the unique $\Z_p$-extension of $\Q$ which is contained in $\Q(\mu_{p^\infty})$. We refer to $\Q_\infty$ as the \emph{cyclotomic $\Z_p$-extension} of $\Q$. Set $\Gamma:=\op{Gal}(\Q_\infty/\Q)$ and $\Gamma_n:=\Gamma/\Gamma^{p^n}$. Let $\Q_n/\Q$ be the extension contained in $\Q_\infty$ such that $[\Q_n:\Q]=p^n$, we find that $\Gamma_n=\op{Gal}(\Q_n/\Q)$. The Iwasawa algebra is the completed group algebra 
\[\Lambda:=\varprojlim_n \cO[\Gamma_n],\] where the inverse limit is taken with respect to natural quotient maps \[\pi_{m,n}:\cO[\Gamma_m]\rightarrow \cO[\Gamma_n]\] for $m\geq n$. Choose a topological generator $\gamma\in \Gamma$ and set $T:=\gamma-1\in \Lambda$. We identify $\Lambda$ with the formal power series ring $\cO\llbracket T\rrbracket$. 

\par In order to define the Bloch--Kato (resp. Greenberg) Selmer group associated to $A$ over $\Q_n$, we must first introduce local conditions. Given a prime $v$ of $\Q_n$, let $\op{I}_{n, v}$ denote the inertia subgroup of $\op{G}_{\Q_{n, v}}$. Let $\ell\neq p$ be a prime number and $v|\ell$ be a prime of $\Q_n$. For $\ast\in \{V, T, A\}$, define 
\[H^1_{\op{nr}}(\Q_{n, v}, \ast):=\op{ker}\left\{ H^1(\Q_{n, v}, \ast)\longrightarrow H^1(\op{I}_{n, v}, \ast)\right\}.\] The quotient map $V\rightarrow A$ induces a map of cohomology groups
\[H^1(\Q_{n, v}, V)\rightarrow H^1(\Q_{n, v}, A); \] we set 
\[\begin{split}
    & H^1_f(\Q_{n, v}, V):=H^1_{\op{nr}}(\Q_{n, v}, V), \\
    & H^1_f(\Q_{n, v}, A):=\op{Im}\left\{ H^1_{\op{nr}}(\Q_{n, v}, V)\rightarrow H^1(\Q_{n, v}, A)\right\}.
\end{split}\]
It is easy to see that $H^1_f(\Q_{n, v}, A)$ is contained in $H^1_{\op{nr}}(\Q_{n, v}, A)$. This follows from the commutativity of the square

\[
\begin{tikzcd}
H^1(\Q_{n, v}, V) \arrow{r}{} \arrow{d}{} & H^1(\op{I}_{n, v}, V) \arrow{d}{} \\
H^1(\Q_{n, v}, A)\arrow{r}{} & H^1(\op{I}_{n, v}, A).
\end{tikzcd}
\]
For a prime $\ell\neq p$, it follows from \cite[Lemma 1.3.5 (i)]{RubinES} that $H^1_f(\Q_\ell, A)$ is the divisible part of $H^1_{\op{nr}}(\Q_\ell, A)$. In particular, the index $[H^1_{\op{nr}}(\Q_\ell, A): H^1_f(\Q_\ell, A)]$ is finite.

\begin{definition}\label{defn of Tamagawa}At a prime $\ell\neq p$, the Tamagawa number of $A$ at $\ell$ is defined as follows
\[c_\ell(A):=[H^1_{\op{nr}}(\Q_\ell, A): H^1_f(\Q_\ell, A)] .\]
\end{definition}
On the other hand, if $v|p$ is a prime of $\Q_n$, define the local condition at $v$ as follows 
\[\begin{split}
    & H^1_f(\Q_{n, v}, V):=\op{ker}\left\{ H^1(\Q_{n, v}, V)\rightarrow H^1(\Q_{n, v}, V\otimes_{\Q_p} \mathbf{B}_{\op{cris}})\right\};\\
    & H^1_f(\Q_{n, v},A):=\op{im}\left\{ H^1_f(\Q_{n, v}, V)\rightarrow  H^1(\Q_{n, v}, A) \right\}.
\end{split}\]

\begin{definition}Define the Bloch--Kato Selmer group of $A$ over $\Q_n$ by
\[\op{Sel}_{\op{BK}}(A/\Q_n):=\op{ker}\left\{ H^1(\Q_n, A)\longrightarrow \prod_{v} \frac{H^1(\Q_{n, v}, A)}{H^1_f(\Q_{n, v}, A)}\right\},\]
where $v$ ranges over the primes of $\Q_n$. The Bloch--Kato Selmer group over $\Q_\infty$ is defined to be the following direct limit \[\op{Sel}_{\op{BK}}(A/\Q_\infty):=\varinjlim_n \op{Sel}_{\op{BK}}(A/\Q_n).\]
\end{definition}

\par We come to the notion of the Greenberg Selmer group. Given a prime $v$ of $\Q_n$ which lies above $p$, consider the composite
\[H^1(\Q_{n, v}, A)\rightarrow H^1(\op{I}_{n, v}, A)\rightarrow H^1(\op{I}_{n, v}, A^-)\] where the first map is restriction to $\op{I}_{n, v}$ and the second map is induced by the natural quotient $A\rightarrow A^-$. For a prime $v|p$ of $\Q_n$, the local Greenberg condition is defined as follows
\[H^1_{\op{ord}}(\Q_{n, v}, A):=\op{ker}\left\{ H^1(\Q_{n, v}, A)\rightarrow H^1(\op{I}_{n, v}, A^-)\right\}.\]

\begin{definition}
    With respect to notation above, define
    \[\op{Sel}_{\op{Gr}}(A/\Q_n):=\op{ker}\left\{ H^1(\Q_n, A)\longrightarrow \prod_{v\nmid p} \frac{H^1(\Q_{n, v}, A)}{H^1_{\op{nr}}(\Q_{n, v}, A)}\times \prod_{v|p} \frac{H^1(\Q_{n, v}, A)}{H^1_{\op{ord}}(\Q_{n, v}, A)}\right\}.\]
    The Greenberg Selmer group over $\Q_\infty$ is defined as the following direct limit
    \[\op{Sel}_{\op{Gr}}(A/\Q_\infty):=\varinjlim_n \op{Sel}_{\op{Gr}}(A/\Q_n).\]
\end{definition}
For ease of notation, let 
\[J_\ell^{\op{Gr}}(A/\Q_n):=\begin{cases}
    \bigoplus_{v|\ell} \frac{H^1(\Q_{n, v}, A)}{H^1_{\op{nr}}(\Q_{n, v}, A)} & \text{ if }\ell\neq p;\\
    \bigoplus_{v|\ell} \frac{H^1(\Q_{n, v}, A)}{H^1_{\op{ord}}(\Q_{n, v}, A)}& \text{ if }\ell=p;
\end{cases}\]
and set $J_\ell^{\op{Gr}}(A/\Q_\infty):=\varinjlim_n J_\ell^{\op{Gr}}(A/\Q_n)$ with respect to natural restriction maps. Let $S$ be the set of primes $\ell$ which divide $Np$. The Greenberg Selmer group over $\Q_\infty$ fits into the following exact sequence
\[0\rightarrow \op{Sel}_{\op{Gr}}(A/\Q_\infty)\rightarrow H^1(\Q_S/\Q_\infty, A)\rightarrow \bigoplus_{\ell\in S} J_\ell^{\op{Gr}}(A/\Q_\infty). \]
Likewise, we set 
\[J_\ell^{\op{BK}}(A/\Q_n):= \bigoplus_{v|\ell} \frac{H^1(\Q_{n, v}, A)}{H^1_{f}(\Q_{n, v}, A)} \]
and \[J_\ell^{\op{BK}}(A/\Q_\infty):=\varinjlim_n J_\ell^{\op{BK}}(A/\Q_n).\] Then the Bloch-Kato Selmer group fits into a natural short exact sequence 
\[0\rightarrow \op{Sel}_{\op{BK}}(A/\Q_\infty)\rightarrow H^1(\Q_S/\Q_\infty, A)\rightarrow \bigoplus_{\ell\in S} J_\ell^{\op{BK}}(A/\Q_\infty).\]
For all $n$, there is a natural inclusion
\[\iota_n: \op{Sel}_{\op{BK}}(A/\Q_n)\hookrightarrow \op{Sel}_{\op{Gr}}(A/\Q_n),\] see \cite[Proposition 4.1, part 1]{Ochiai}. In fact, for the prime $v$ of $\Q_n$ such that $v|p$, one has that $H^1_f (\Q_{n, v}, A)$ is a finite index subgroup of $H^1_{\op{ord}}(\Q_{n, v}, A)$, cf. \cite[Theorem 3]{Flach}.
Recall that for $v\nmid p$, $H^1_f (\Q_{n, v}, A)$ is the divisible part of $H^1_{\op{nr}} (\Q_{n, v}, A)$. Thus the Bloch--Kato conditions are contained in the Greenberg conditions. Passing to the direct limit, we find that there is an inclusion
\[\iota_\infty:\op{Sel}_{\op{BK}}(A/\Q_\infty)\hookrightarrow \op{Sel}_{\op{Gr}}(A/\Q_\infty).\]
Let $\alpha: \op{Sel}_{\op{BK}}(A/\Q)\rightarrow \op{Sel}_{\op{Gr}}(A/\Q_\infty)^\Gamma$ be the composite of maps 
\begin{equation}\label{defn of alpha}\op{Sel}_{\op{BK}}(A/\Q)\rightarrow \op{Sel}_{\op{Gr}}(A/\Q)\rightarrow \op{Sel}_{\op{Gr}}(A/\Q_\infty)^\Gamma.\end{equation}

\subsection{Iwasawa modules and invariants}\label{s 2.2}
\par We introduce the algebraic Iwasawa invariants associated to the Greenberg Selmer group $\op{Sel}_{\op{Gr}}(A/\Q_\infty)$. Given $\Lambda$-modules $M$ and $M'$ a \emph{pseudo-isomorphism} is a map of $\Lambda$-modules $\phi: M\rightarrow M'$ whose kernel and cokernel are both finite. A polynomial in $\cO\llbracket T\rrbracket$ is \emph{distinguished} if it is monic and its non-leading coefficients are in $(\varpi)$. Let $M$ be a finitely generated and torsion module over $\Lambda$. By the structure theory of such modules, there is a pseudo-isomorphism 
\begin{equation}\label{pseudo-iso}M\longrightarrow \left( \bigoplus_{i=1}^s \Lambda/(p^{\mu_i})\right)\oplus \left( \bigoplus_{j=1}^t \Lambda/(f_j(T)),\right)\end{equation}where $f_j(T)$ are distinguished polynomials. The $\mu$-invariant and $\lambda$-invariant of $M$ are defined as follows
\begin{equation}\label{structure decomposition}\mu(M):=\sum_{i=1}^s \mu_i\text{ and }\lambda(M):=\sum_{j=1}^t \op{deg}f_j(T). \end{equation} 
We set $\mu(M)=0$ (resp. $\lambda(M)=0$) if $s=0$ (resp. $t=0$). These quantities are well defined, i.e., independent of the choice of decomposition \eqref{pseudo-iso}. Let $g(M)$ be the minimal number of generators of $M$ as a $\Lambda$-module. It follows from Nakayama's lemma that 
\[g(M)=\dim_{\kappa} \left(\frac{M}{(\varpi, T)M}\right).\] Given a $\Lambda$-module $M$, let $M^\vee:=\op{Hom}_{\Z_p}(M, \Q_p/\Z_p)$ denote its Pontryagin dual. As is well known, it is a result of \cite{Katopadichodge} that $\op{Sel}_{\op{Gr}}(A/\Q_\infty)^\vee$ is torsion as a $\Lambda$-module. Denote by $\mu(A)$ (resp. $\lambda(A)$) the $\mu$-invariant (resp. $\lambda$-invariant) of $\op{Sel}_{\op{Gr}}(A/\Q_\infty)^\vee$. We mention here in passing that the Greenberg Selmer groups considered in this article are shown to be cotorsion over $\Lambda$, independent of Kato's result (see Remark \ref{remark on cotorsion}). Set $g(A)$ to denote the number of generators of $\op{Sel}_{\op{Gr}}(A/\Q_\infty)^\vee$. We note that 
\[g(A)=\dim_\kappa \left(\frac{\op{Sel}_{\op{Gr}}(A/\Q_\infty)^\vee}{(\varpi, T)\op{Sel}_{\op{Gr}}(A/\Q_\infty)^\vee}\right)=\dim_{\kappa} \left(\op{Sel}_{\op{Gr}}(A/\Q_\infty)^\Gamma[\varpi]\right).\]
\section{Selmer ranks: local and global theory}\label{s 3}
\par In this section, we study the relationship between Selmer ranks and Iwasawa theory. We begin by recalling a rather general version of Cassels' theorem. 
\subsection{Generalities on Selmer groups and Cassels' theorem}
\par The discussion in this subsection summarizes \cite[pp. 99-102]{GreenbergCetaro}. Let $p$ be an odd prime number and $F$ be a number field. Let $\Sigma$ be a finite set of primes of $F$ containing the primes of $F$ that lie above $p$. Let $F_\Sigma$ be the maximal extension of $F$ contained in $\bar{\Q}$ in which all primes $v\notin \Sigma$ are unramified. Let $M$ be a $\op{Gal}(F_\Sigma/F)$-module whose underlying group is isomorphic to $(\Q_p/\Z_p)^d$. For each prime $v\in \Sigma$, let $L_v$ be a divisible subgroup of $H^1(F_v, M)$ and consider the associated Selmer group 
\[S_M(F):=\op{ker}\left\{H^1(F_\Sigma/F, M)\xrightarrow{\Phi_M}\bigoplus_{v\in \Sigma} \frac{H^1(F_v, M)}{L_v}\right\}.\]
This is a discrete $p$-primary group which is cofinitely generated over $\Z_p$. Now let $T^*:=\op{Hom}_{\Z_p}(M, \mu_{p^\infty})$ which is a free $\Z_p$-module of rank $d$. For $v\in \Sigma$, let $U_v^*$ be the orthogonal complement of $L_v$ and consider the Selmer group 
\[S_{T^*}(F):=\op{ker}\left\{H^1(F_\Sigma/F, T^*)\xrightarrow{\Phi_{T^*}}\bigoplus_{v\in \Sigma} \frac{H^1(F_v, T^*)}{U_v^*}\right\}.\]
This is a finitely generated $\Z_p$-module. Let $V^*:=T^*\otimes_{\Z_p} \Q_p$ and $M^*:=V^*/T^*$. For $v\in \Sigma$ we find that $U_v^*$ generates a $\Q_p$-subspace of $H^1(F_v, V^*)$ under the map 
\[H^1(F_v, T^*)\rightarrow H^1(F_v, V^*).\] We let $L_v^*$ be the image of this subspace with respect to the map 
\[H^1(F_v, V^*)\rightarrow H^1(F_v, M^*).\] Let $S_{M^*}(F)$ be the associated Selmer group
\[S_{M^*}(F):=\op{ker}\left\{H^1(F_\Sigma/F, M^*)\xrightarrow{\Phi_{M^*}}\bigoplus_{v\in \Sigma} \frac{H^1(F_v, M^*)}{L_v^*}\right\}.\]
Assume that $H^0(F_v, M^*)$ is finite for at least one $v\in \Sigma$. Let $m^*$ be the $\Z_p$-corank of $S_{M^*}(F)$, then, the cokernel of $\Phi_M$ has corank $\leq m^*$. Moreover, we find that 
\begin{equation}\label{dim cok formula}\op{dim}_{\F_p} (\op{cok}\Phi_M)[p]\leq m^* +\op{dim}_{\F_p} H^0(F, M^*[p]). \end{equation}
 For further details, we refer to \cite[Proposition 4.13]{GreenbergCetaro}. 

\subsection{Comparing Bloch--Kato Selmer groups} Now we specialize our discussion to $M:=A$ and $F:=\Q$. Assume moreover that $\cO=\Z_p$ and thus $\cK=\Q_p$ and that the Selmer conditions on $A$ are the Bloch--Kato conditions $H^1_f$. Since $V$ is self-dual (i.e., $\op{Hom}(V, \Q_p(1))\simeq V$), we find that $M^*=A$. Since the Bloch-Kato Selmer conditions are also self-dual, we have that
\begin{equation}\label{self dual bloch kato selmer}S_{M}(\Q)=S_{M^*}(\Q)=\op{Sel}_{\op{BK}}(A/\Q),\end{equation} cf. \cite[p. 457]{LongoVigni}.
Recall that $S$ is the set of primes $\ell|N p$. Since $H^1_f(\Q_\ell, A)$ is contained in $H^1_{\op{nr}}(\Q_\ell, A)$ (resp. $H^1_{\op{ord}}(\Q_\ell, A)$) for $\ell\neq p$ (resp. $\ell=p$), we get a natural map 
\[J_\ell^{\op{BK}}(A/\Q)\rightarrow J_\ell^{\op{Gr}}(A/\Q).\] For every prime $\ell$, set 
\[h_{\ell,A}: J_\ell^{\op{BK}}(A/\Q)\longrightarrow J_\ell^{\op{Gr}}(A/\Q_\infty)^\Gamma,\] which is the composite of natural maps
\[J_\ell^{\op{BK}}(A/\Q)\rightarrow J_\ell^{\op{Gr}}(A/\Q)\rightarrow J_\ell^{\op{Gr}}(A/\Q_\infty)^\Gamma.\]
We set $\delta_\ell(A):=\op{dim}_{\F_p} (\op{ker}h_{\ell,A}) [p]$. Let 
\[h_A : \bigoplus_{\ell\in S} J_\ell^{\op{BK}}(A/\Q)\longrightarrow \bigoplus_{\ell\in S} J_\ell^{\op{Gr}}(A/\Q_\infty)^\Gamma \] be the direct sum $\bigoplus_{\ell\in S} h_{\ell,A}$. Then note that 
\[\op{dim}_{\F_p} (\op{ker}h_A)[p]=\sum_{\ell\in S} \op{dim}_{\F_p} (\op{ker}h_{\ell,A})[p]=\sum_{\ell\in S} \delta_\ell(A). \]
    The Bloch--Kato Selmer group $\op{Sel}_{\op{BK}}(A/\Q)$ and the Greenberg Selmer group $\op{Sel}_{\op{Gr}}(A/\Q_\infty)$ fit into a commutative diagram
\begin{equation}\label{BK exact diagram}
\begin{tikzcd}[column sep = small, row sep = large]
0\arrow{r} & \op{Sel}_{\op{BK}}(A/\Q) \arrow{r} \arrow{d}{\alpha} & H^1\left(\Q_S/\Q,A\right) \arrow{r}{\Phi} \arrow{d}{g} & \bigoplus_{\ell\in S} J_\ell^{\op{BK}}(A/\Q) \arrow{d}{h_A} &\\
0\arrow{r} & \op{Sel}_{\op{Gr}}(A/\Q_\infty)^{\Gamma} \arrow{r} & H^1\left(\Q_S/\Q_\infty,A\right)^{\Gamma} \arrow{r} &\left(\bigoplus_{\ell\in S} J_\ell^{\op{Gr}}(A/\Q_\infty)\right)^\Gamma,
\end{tikzcd}\end{equation}
    where the vertical map $\alpha$ is induced by $g$, cf. \eqref{defn of alpha}.
\begin{lemma}
    With respect to notation above, the following assertions hold.
    \begin{enumerate}
        \item For any prime $\ell\neq p$, we have that $\#\op{ker} h_{\ell,A}= c_\ell(A)$. 
        \item The map $h_{\ell,A}$ is injective for $\ell\notin S$.
    \end{enumerate}
\end{lemma}
\begin{proof}
\par We first consider the case when $\ell\neq p$. We analyze the kernels of the two maps $\pi_\ell: J_\ell^{\op{BK}}(A/\Q)\rightarrow J_\ell^{\op{Gr}}(A/\Q)$ and $\pi_\ell':J_\ell^{\op{Gr}}(A/\Q)\rightarrow J_\ell^{\op{Gr}}(A/\Q_\infty)$. The composite of $\pi_\ell$ and $\pi_\ell'$ is $h_{\ell, A}$. We compute $\# \op{ker}\pi_\ell$ and show that $\pi_\ell'$ is injective. The kernel of the map 
    \[\pi_\ell: \frac{H^1(\Q_\ell, A)}{H^1_f(\Q_\ell, A)}\rightarrow \frac{H^1(\Q_\ell, A)}{H^1_{\op{nr}}(\Q_\ell, A)}\] is the quotient
    \[\op{ker}\pi_\ell= \frac{H^1_{\op{nr}}(\Q_\ell, A)}{H^1_f(\Q_\ell, A)},\]
    and hence $\# \op{ker}\pi_\ell=c_\ell(A)$, see Definition \ref{defn of Tamagawa}.
    Since $\ell\neq p$, we find that $\Q_\infty/\Q$ is unramified at $\ell$. It is thus easy to see that the map \[\pi_\ell':\frac{H^1(\Q_\ell, A)}{H^1_{\op{nr}}(\Q_\ell, A)}\rightarrow \bigoplus_{v|\ell} \frac{H^1(\Q_{\infty, v}, A)}{H^1_{\op{nr}}(\Q_{\infty, v}, A)}\] is injective. Thus the kernel of the map $h_{\ell,A}$ coincides with the kernel of $\pi_\ell$ and we have shown that for $\ell\neq p$, then
    \[\# \op{ker}h_{\ell,A}=c_\ell(A).\] This completes the proof of (1)
    \par Next we prove part (2). For $\ell\notin S$ one finds that $c_\ell(A)=1$ by \cite[Lemma 1.3.5 (iv)]{RubinES}. Thus it follows from part (1) that $h_{\ell, A}$ is injective.
\end{proof}
For a prime $\ell\neq p$, we describe the quantity $c_\ell(A)$ in more detail. We denote by $\sigma_\ell$ a Frobenius element at $\ell$. 
\begin{lemma}\label{formula for the Tamagawa number}
    Let $\ell\neq p$ and set $\mathcal{A}:=A^{\op{I}_\ell}/(A^{\op{I}_\ell})_{\op{div}}$. Then we have that 
    \[\op{ker}h_{\ell,A}=\left(\mathcal{A}/(\sigma_\ell-1)\mathcal{A}\right)\text{ and }c_\ell(A)=\#\left(\mathcal{A}/(\sigma_\ell-1)\mathcal{A}\right).\]
\end{lemma}
\begin{proof}
    The result above is \cite[Lemma 1.3.5 (iii)]{RubinES}.
\end{proof}

\begin{lemma}\label{explicit formula for tamagawa}
    With respect to the above notation, the following assertions hold.
    \begin{enumerate}
        \item If $\ell\neq p$, we have that $\delta_\ell(A)\leq 2$,
        \item $\delta_p(A)\leq 6$.
    \end{enumerate}
\end{lemma}
\begin{proof}
    First consider the case when $\ell\neq p$. Note that $A\simeq (\Q_p/\Z_p)^2$ and it follows from Lemma \ref{formula for the Tamagawa number} that $\op{ker} h_{\ell, A}$ is a quotient of $A^{\op{I}_\ell}$. It thus follows that 
    \[\delta_\ell(A)=\dim_{\F_p}\left(\mathcal{A}/(\sigma_\ell-1)\mathcal{A}\right)[p]\leq 2. \]
    \par Next, we treat $\ell=p$. It is clear that $\dim (\op{ker} h_{p, A})[p]\leq \dim H^1(\Q_p, A)[p]$. From the Kummer sequence, we obtain the exact sequence
    \[0\rightarrow \frac{H^0(\Q_p, A)}{p H^0(\Q_p, A)} \rightarrow H^1(\Q_p, A[p])\rightarrow H^1(\Q_p, A)[p]\rightarrow 0.\] 
    From the Euler characteristic formula, we find that
    \[\dim H^1(\Q_p, A[p])=\dim H^0(\Q_p, A[p])+\dim H^2(\Q_p, A[p])+\dim A[p].\]
It follows from local duality that $\dim H^2(\Q_p, A[p])=\dim H^0(\Q_p, A^*[p])$. Here, $A^*$ is the Tate-dual of $A$. We note that $A=A^*$ and thus, we find that 
\[\dim H^1(\Q_p, A[p])\leq 3\dim A[p]=6.\qedhere\] 
\end{proof}
\begin{remark}
    If $\dim H^0(\Q_p, A[p])\leq 1$ as will be the case in our applications, then the bound in part (2) of the lemma can be improved to $\delta_p(A)\leq 4$.
\end{remark}
\begin{proposition}\label{g(A)>= sum delta l}
    Assume that $H^0(\Q,A)=0$ and that there is a prime $\ell$ for which $H^0(\Q_\ell, A)$ is finite. Then, with respect to notation above, we have that
    \[g(A)=\op{dim}_{\F_p}\left(\op{Sel}_{\op{Gr}}(A/\Q_\infty)^\Gamma[p]\right)\geq \sum_{\ell\in S} \delta_\ell(A).\]
\end{proposition}
\begin{proof}
    The diagram \eqref{BK exact diagram} can be spliced into two commutative diagrams with exact rows:
\begin{equation}\label{diagram1}
\begin{tikzcd}[column sep = small, row sep = large]
0\arrow{r} & \op{Sel}_{\op{BK}}(A/\Q) \arrow{r} \arrow{d}{\alpha} & H^1\left(\Q_S/\Q,A\right) \arrow{r} \arrow{d}{g} & \op{im} \Phi \arrow{r} \arrow{d}{h'} & 0\\
0\arrow{r} & \op{Sel}_{\op{Gr}}(A/\Q_\infty)^{\Gamma} \arrow{r} & H^1\left(\Q_S/\Q_\infty,A\right)^{\Gamma} \arrow{r} &\left(\bigoplus_{\ell\in S} J_\ell^{\op{Gr}}(A/\Q_\infty)\right)^\Gamma
\end{tikzcd}\end{equation}
and 
\begin{equation}\label{diagram2}
\begin{tikzcd}[column sep = small, row sep = large]
0\arrow{r} & \op{im} \Phi \arrow{r} \arrow{d}{h'} & \bigoplus_{\ell\in S} \frac{H^1(\Q_\ell, A)}{H^1_f(\Q_\ell, A)}\arrow{r} \arrow{d}{h_A} & \op{cok}\Phi \arrow{r} 
 \arrow{d}{j} & 0\\
0\arrow{r} & \left(\bigoplus_{\ell\in S} J_\ell^{\op{Gr}}(A/\Q_{\infty})\right)^\Gamma \arrow{r} & \left(\bigoplus_{\ell\in S} J_\ell^{\op{Gr}}(A/\Q_{\infty})\right)^\Gamma \arrow{r} & 0,
\end{tikzcd}\end{equation}
where $h'$ is the restriction of $h_A$ to the image of $\Phi$. From the inflation-restriction sequence, we have
\[H^1(\Q_\infty/\Q, H^0(\Q_\infty, A))\rightarrow H^1(\Q_S/\Q, A)\xrightarrow{g} H^1(\Q_S/\Q_\infty, A)^\Gamma \rightarrow H^2(\Q_\infty/\Q, H^0(\Q_\infty, A)).\]
Note that $H^0(\Q, A)=0$ and since $\Q_\infty/\Q$ is pro-$p$, it follows that $H^0(\Q_\infty, A)=0$ (see \cite[Corollary 1.6.13]{NSW}). Thus $g$ is an isomorphism. Thus from \eqref{diagram1}, the map $\alpha$ is injective with cokernel isomorphic to the kernel of $h'$. We thus have a short exact sequence:
    \begin{equation}\label{alpha exact sequence}0\rightarrow \op{Sel}_{\op{BK}}(A/\Q)\xrightarrow{\alpha} \op{Sel}_{\op{Gr}}(A/\Q_\infty)^\Gamma \rightarrow \op{ker}h'\rightarrow 0.\end{equation} On the other hand, from \eqref{diagram2}, we obtain the left exact sequence:
    \begin{equation}\label{h' and h exact sequence}0\rightarrow \op{ker} h' \rightarrow \op{ker} h_A \rightarrow \op{cok}\Phi.\end{equation}
Write \[\op{Sel}_{\op{BK}}(A/\Q)=(\Q_p/\Z_p)^{r(A)} \oplus T,\] where $r(A)$ is the $\Z_p$-corank of the Bloch-Kato Selmer group and $T$ is a finite abelian $p$-group. We find that 
     \begin{equation}\label{Sel BK mod p and r(A)}
     \begin{split}\dim_{\F_p} \op{Sel}_{\op{BK}}(A/\Q)[p] 
     = & \dim_{\F_p}T[p]+ r(A)
     =  \dim_{\F_p} (T/p T)+r(A), \\
     = & \dim_{\F_p} \left(\frac{\op{Sel}_{\op{BK}}(A/\Q)}{p \op{Sel}_{\op{BK}}(A/\Q)}\right)+ r(A).
\end{split}\end{equation} By assumption, there is a prime $\ell$ such that $H^0(\Q_\ell, A)$ is finite. By \eqref{dim cok formula} and \eqref{self dual bloch kato selmer}, we have that \begin{equation}\label{cok Phi[p]}\dim_{\F_p}(\op{cok} \Phi) [p]\leq r(A),\end{equation}
since $H^0(\Q, A)=0$.
Applying the snake lemma to \eqref{alpha exact sequence}, we arrive at an exact sequence
     \begin{equation}\label{exact sequence 4 term}0\rightarrow \op{Sel}_{\op{BK}}(A/\Q)[p]\rightarrow \op{Sel}_{\op{Gr}}(A/\Q_\infty)^\Gamma [p]\rightarrow \op{ker}h'[p]\rightarrow \frac{\op{Sel}_{\op{BK}}(A/\Q)}{p \op{Sel}_{\op{BK}}(A/\Q)}.\end{equation}
    From \eqref{h' and h exact sequence}, \eqref{Sel BK mod p and r(A)}, \eqref{cok Phi[p]}, \eqref{exact sequence 4 term} we arrive at the following:
     \begin{align*}
& \dim_{\F_p}\op{Sel}_{\op{Gr}}(A/\Q_\infty)^\Gamma[p] \\
& \geq \dim_{\F_p} \op{Sel}_{\op{BK}}(A/\Q)[p] + \dim_{\F_p} \left(\ker h'\right)[p] - \dim_{\F_p} \left(\frac{\op{Sel}_{\op{BK}}(A/\Q)}{p\op{Sel}_{\op{BK}}(A/\Q)}\right)\\
& = \dim_{\F_p} (\ker h')[p]+r(A)\\
& \geq \dim_{\F_p} (\ker h')[p]+\dim_{\F_p}(\op{cok}\Phi)[p]\\
& \geq \dim_{\F_p} (\ker h_A)[p]\\
& = \sum_{\ell\in S} \op{dim}_{\F_p}( \op{ker}h_{\ell,A})[p]=\sum_{\ell\in S} \delta_\ell(A). \qedhere
\end{align*}
\end{proof}

\par Now suppose that $\bar{\rho}_T$ is reducible and of the form $\mtx{\varphi\bar{\psi}}{\ast}{0}{\bar{\psi}}$ with respect to some basis $e_1$, $e_2$ of $T$. Then there is the lattice $T'$ generated by $e_1':=p^{-1} e_1$ and  $e_2':=e_2$ which is also Galois stable and $\bar{\rho}_{T'}$ is of the form $\mtx{\varphi\bar{\psi}}{0}{\ast}{\bar{\psi}}$. The inclusion map $\phi: T\rightarrow T'$ has cokernel isomorphic to $\F_p(\varphi\bar{\psi})$. Let $A$ (resp. $A'$) denote the $p$-divisible group associated to $T$ (resp. $T'$), i.e., $A:=V/T$ (resp. $A':=V/T'$). Then, $\phi: A\rightarrow A'$ has kernel isomorphic to $A_1:=\F_p(\varphi\bar{\psi})$. Denote by $\phi^\vee: A'\rightarrow A$ the \emph{dual isogeny}, induced by the composite 
\[\phi^\vee: T'\hookrightarrow p^{-1} T\xrightarrow{\sim} T.\] 
We now show that under some conditions on the characters $\varphi$ and $\bar{\psi}$ the $\mu$-invariant of the Greenberg Selmer group $\op{Sel}_{\op{Gr}}(A/\Q_\infty)$ is $0$.
\begin{lemma}\label{basic lemma on mu}
    Let $N$ be a cofinitely generated over the Iwasawa algebra $\Lambda=\Z_p\llbracket T\rrbracket$. Recall that $N^\vee$ denotes the Pontryagin dual of $N$. Then the following are equivalent:
    \begin{enumerate}
        \item $N$ is cotorsion over $\Lambda$ and $\mu(N^\vee)=0$,
        \item $N$ is cofinitely generated as a $\Z_p$-module,
        \item $N[p]$ is finite.
    \end{enumerate}
\end{lemma}
\begin{proof}
    That (2) and (3) are equivalent is clear. We show that (1) and (3) are equivalent. We denote by $M$ the Pontryagin dual of $N$. Note that if $M$ is pseudo-isomorphic to $M'$, then $M$ is finitely generated as a $\Z_p$-module if and only if $M'$ is finitely generated as a $\Z_p$-module. Thus by replacing $M$ by a pseudo-isomorphic module if necessary, we have an isomorphism of $\Lambda$-modules
    \[M\xrightarrow{\sim}\Lambda^r\oplus \left( \bigoplus_{i=1}^s \Lambda/(p^{\mu_i})\right)\oplus \left( \bigoplus_{j=1}^t \Lambda/(f_j(T))\right),\]
    where $f_j(T)$ are distinguished polynomials. Note that $r=0$ if and only if $M$ is torsion over $\Lambda$. We find that 
    \[\begin{split} M/pM & \xrightarrow{\sim} \left(\Lambda/(p)\right)^{s+r}\oplus \left( \bigoplus_{j=1}^t \Lambda/(p, f_j(T))\right)\\
    & \xrightarrow{\sim} \F_p\llbracket T\rrbracket^{s+r}\oplus \left( \bigoplus_{j=1}^t \F_p[T]/(T^{\op{deg}(f_j)})\right).
    \end{split}\]
    Therefore $M/pM$ is finite if and only if $s+r=0$. Note that $s=0$ if and only if $\mu(M)=0$. On the other hand, $N[p]$ is dual to $M/pM$. Thus we have shown that $N[p]$ is finite if and only if $M$ is torsion as a $\Lambda$-module and $\mu(M)=0$, which completes the proof.
\end{proof}

\begin{proposition}\label{mu=0 propn}
    Assume that both characters $\varphi$ and $\bar{\psi}$ are odd. Moreover assume that $\varphi$ is ramified at $p$ and $\bar{\psi}$ is unramified at $p$. Then we have that $\mu(A)=0$.
\end{proposition}
\begin{proof}
    The result is proven for elliptic curves with good ordinary reduction at an odd prime $p$ \cite[Proposition 5.10]{GreenbergCetaro}. We generalize this argument to our setting. It follows from Lemma \ref{basic lemma on mu} that $\mu(A)=0$ if and only if $\op{Sel}_{\op{Gr}}(A/\Q_\infty)[p]$ is finite. Set $\varphi_1:=\varphi\bar{\psi}$ and $\varphi_2:=\bar{\psi}$. The module $A[p]$ sits in a short exact sequence 
    \[0\rightarrow A_1 \rightarrow A[p]\rightarrow A_2\rightarrow 0, \]
    where $A_i:=\F_p(\varphi_i)$. From the above, we have the exact sequence
    \begin{equation}\label{a,b ses}H^1(\Q_S/\Q_\infty, A_1)\xrightarrow{a} H^1(\Q_S/\Q_\infty, A[p])\xrightarrow{b}H^1(\Q_S/\Q_\infty, A_2).\end{equation}We note that since $\varphi\bar{\psi}$ is ramified at $p$, $H^0(\Q, A_1)=0$. On the other hand, since $\bar{\psi}$ is odd, $H^0(\Q, A_2)=0$. Since $\Q_\infty/\Q$ is a pro-$p$ extension, it follows from \cite[Corollary 1.6.13]{NSW} that $H^0(\Q_\infty, A_i)=0$ for $i=1, 2$. It follows therefore that $H^0(\Q_\infty, A[p])=0$. Consider the Kummer sequence 
    \[0\rightarrow A[p]\rightarrow A\xrightarrow{\times p} A\rightarrow 0\]
    and the associated long exact sequence in cohomology:
    \[\dots \rightarrow H^0(\Q_\infty, A)\rightarrow H^1(\Q_S/\Q_\infty, A[p])\rightarrow H^1(\Q_S/\Q_\infty, A)\xrightarrow{\times p} H^1(\Q_S/\Q_\infty, A).\]
    Since $H^0(\Q_\infty, A[p])=0$ it follows that $H^0(\Q_\infty, A)=0$. As a result, we have an isomorphism:
    \[H^1(\Q_S/\Q_\infty, A[p])\xrightarrow{\sim} H^1(\Q_S/\Q_\infty, A)[p].\]
    Thus we can view $\op{Sel}_{\op{Gr}}(A/\Q_\infty)[p]$ as a subgroup of $H^1(\Q_S/\Q_\infty, A[p])$. If $\op{Sel}_{\op{Gr}}(A/\Q_\infty)[p]$ is infinite then it follows that either:
    \begin{itemize}
        \item $\mathscr{A}:=\op{Image}(a)\cap \op{Sel}_{\op{Gr}}(A/\Q_\infty)[p]$ is infinite, or, 
        \item $\mathscr{B}:=b\left( \op{Sel}_{\op{Gr}}(A/\Q_\infty)[p]\right)$ is infinite.
    \end{itemize}
   Thus in order to complete the proof it suffices to show that both $\mathscr{A}$ and $\mathscr{B}$ are finite.
   \par Since the character $\varphi_1$ is even, it follows from \cite[Lemma 5.9]{GreenbergCetaro} that $H^1(\Q_S/\Q_\infty, A_1)$ is finite and we deduce that $\mathscr{A}$ is finite. We show that $\mathscr{B}$ is finite. Since $\varphi$ is ramified and $\bar{\psi}$ is unramified, $A^-[p]$ is the unique $1$-dimensional unramified quotient of $A[p]$. Thus we may identify $A_2$ with $A^-[p]$. From our identifications, the Greenberg Selmer group $\op{Sel}_{\op{Gr}}(A/\Q_\infty)[p]$ consists of $\alpha\in H^1(\Q_S/\Q_\infty, A[p])$ such that 
    \begin{itemize}
        \item $\alpha$ is unramified at all primes $w\nmid p$ of $\Q_\infty$, 
        \item $b(\alpha)$ is unramified at all primes of $\Q_\infty$.
    \end{itemize}
    Let $H^1_{\op{nr}}(\Q_S/\Q_\infty, \F_p(\bar{\psi}))$ consist of classes in $H^1(\Q_S/\Q_\infty, \F_p(\bar{\psi}))$ which are unramified at all primes of $\Q_\infty$. Since $b(\alpha)$ belongs to $H^1_{\op{nr}}(\Q_S/\Q_\infty, \F_p(\bar{\psi}))$, it suffices to show that $H^1_{\op{nr}}(\Q_S/\Q_\infty, \F_p(\bar{\psi}))$ is finite. Let $K:=\Q(\bar{\psi})$ be the abelian extension of $\Q$ which is cut out by $\bar{\psi}$ and $K_\infty$ the cyclotomic $\Z_p$-extension of $K$. Consider the inflation-restriction sequence:
    \[0\rightarrow H^1\left(\op{Gal}(K_\infty/\Q_\infty),H^0(K_\infty, \F_p(\bar{\psi}))\right)\rightarrow H^1(\Q_S/\Q_\infty, \F_p(\bar{\psi}))\rightarrow H^1(\Q_S/K_\infty, \F_p(\bar{\psi})). \]
    Note that $H^1\left(\op{Gal}(K_\infty/\Q_\infty),H^0(K_\infty, \F_p(\bar{\psi}))\right)$ vanishes as $\op{Gal}(K_\infty/\Q_\infty)$ has order dividing $(p-1)$. Thus in order to show that $H^1_{\op{nr}}(\Q_S/\Q_\infty, \F_p(\bar{\psi}))$ is finite, it suffices to show that $H^1_{\op{nr}}(\Q_S/K_\infty, \F_p(\bar{\psi}))$ is finite. The action of $\op{Gal}(\Q_S/K)$ on $\F_p(\bar{\psi})$ is trivial. Thus, a non-zero element $\beta\in H^1_{\op{nr}}(\Q_S/K_\infty, \F_p(\bar{\psi}))$ gives a surjective homomorphism $\beta:\op{Gal}(\Q_S/K_\infty)\rightarrow \Z/p\Z$ which is unramified at all primes of $K_\infty$. Let $\mathcal{L}/K_\infty$ be the maximal $p$-elementary abelian extension of $K_\infty$ in which all primes of $K_\infty$ are unramified. The homomorphism $\beta$ factors as a surjective map $\beta: \op{Gal}(\mathcal{L}/K_\infty)\rightarrow \Z/p\Z$. Let $\widetilde{\mathcal{L}}$ be the maximal abelian unramified pro-$p$ extension of $K_\infty$, and set $X:=\op{Gal}(\widetilde{\mathcal{L}}/K_\infty)$. A well-known theorem of Ferrero and Washington \cite{FerreroWash} asserts that the Iwasawa $\mu$-invariant of the $\Z_p$-extension $K_\infty/K$ vanishes, i.e., $\mu(X)=0$. Consequently by Lemma \ref{basic lemma on mu}, it follows that $X/p X$ is finite. Identify $X/pX$ with $\op{Gal}(\mathcal{L}/K_\infty)$ to deduce that it is finite. Hence, the number of homomorphisms $\beta: \op{Gal}(\mathcal{L}/K_\infty)\rightarrow \Z/p\Z$ is also finite. Thus we have shown that $\mathscr{B}$ is finite and the proof is complete.
\end{proof}

\begin{remark}\label{remark on cotorsion}
    We observe that the preceding proof and the implication (3) $\implies$ (1) of Lemma \ref{basic lemma on mu} actually establishes that $\operatorname{Sel}_{\operatorname{Gr}}(A/\mathbb{Q}_\infty)$ is cotorsion as a $\Lambda$-module. Moreover, since $A'$ is isogenous to $A$, the isogeny $\phi$ induces a natural map  
\[
\operatorname{Sel}_{\operatorname{Gr}}(A/\mathbb{Q}_\infty) \to \operatorname{Sel}_{\operatorname{Gr}}(A'/\mathbb{Q}_\infty)
\]  
whose kernel and cokernel are cotorsion as $\Lambda$-modules. It follows that $\operatorname{Sel}_{\operatorname{Gr}}(A'/\mathbb{Q}_\infty)$ is also cotorsion as a $\Lambda$-module. Consequently, the results below remain valid independently of Kato’s theorem on torsion-ness of Selmer groups referenced in Section \ref{s 2.2}.
\end{remark}

\begin{proposition}\label{g(X) >= g(Y)/2}
    Let $X$ be a finitely generated and torsion $\Lambda$-module with no finite $\Lambda$-submodule. Let $Y$ be a $\Lambda$-submodule of $X$ such that $Z:=X/Y$ is of exponent $p$. Then, we have that $g(X)\geq g(Y)/2$.
\end{proposition}
\begin{proof}
    For a proof of this result, we refer to \cite[Lemma 5.3]{matsuno}.
\end{proof}

Given a $\Lambda$-module $M$, recall that $M^\vee:=\op{Hom}_{\Z_p}\left(M, \Q_p/\Z_p\right)$. 

\begin{proposition}\label{propn no nontriv}
    Assume that $H^0(\Q, A)=0$. Then $\op{Sel}_{\op{Gr}}(A/\Q_\infty)^\vee$ has no non-trivial finite $\Lambda$-submodules. 
\end{proposition}
\begin{proof}
   It follows from \cite[Proposition 5.15]{LongoVigni} that $H^1(\Q_S/\Q_\infty, A)^\vee$ has no non-trivial finite $\Lambda$-submodules. We note that this result applies in our setting, and does not require $A$ to be crystalline at $p$. In greater detail, the proof of \cite[Proposition 5.15]{LongoVigni} relies on \cite[Lemma 5.11, Lemma 5.12, Lemma 5.13]{LongoVigni}. These results are based purely on local computations and require only the ordinary hypothesis. It then follows verbatim from the proof of \cite[Proposition 4.14]{GreenbergCetaro} that $\op{Sel}_{\op{Gr}}(A/\Q_\infty)^\vee$ has no non-trivial finite $\Lambda$-submodules.
\end{proof}

\begin{lemma}\label{g(A')>=g(A)/2}
    Assume that $H^0(\Q, A)=0$, $H^0(\Q, A')=0$ and that $\mu(A)=0$. With respect to notation above, $g(A')\geq g(A)/2$.
\end{lemma}
\begin{proof}
    The argument is entirely analogous to that of \cite[Corollary 5.4]{matsuno} and we give a sketch of the details here. There is a $\Lambda$-homomorphism 
    \[\Psi:\op{Sel}_{\op{Gr}}(A/\Q_\infty)^\vee \rightarrow \op{Sel}_{\op{Gr}}(A'/\Q_\infty)^\vee\] induced by the isogeny $\phi^\vee: A'\rightarrow A$. Since $H^1(\Q_S/\Q_\infty, \op{ker}\phi^\vee)$ is a $p$-torsion module, it follows that the cokernel of $\Psi$ has exponent $p$. Let $M$ be a finitely generated and torsion $\Lambda$-module. Then it is easy to see that the following are equivalent:
    \begin{itemize}
        \item $M$ is a finitely generated and free $\Z_p$-module,
        \item $M$ has no non-trivial finite $\Lambda$-submodules and $\mu(M)=0$.
    \end{itemize}
    According to Proposition \ref{propn no nontriv}, $\op{Sel}_{\op{Gr}}(A/\Q_\infty)^\vee$ has no non-trivial finite $\Lambda$-submodules. It is assumed that the $\mu$-invariant of $\op{Sel}_{\op{Gr}}(A/\Q_\infty)^\vee$ is $0$. Therefore, $\op{Sel}_{\op{Gr}}(A/\Q_\infty)^\vee$ is a finitely generated and free $\Z_p$-module. In particular, the kernel of $\Psi$ has no non-trivial $p$-torsion. Since $H^2(\Q_S/\Q_\infty, \op{ker}\phi^\vee)$ is $p$-torsion, the natural long exact sequence shows that the cokernel of the map
    \[H^1(\Q_S/\Q_\infty, A')\rightarrow H^1(\Q_S/\Q_\infty, A)\] is $p$-torsion. Taking the dual, it follows that $\op{ker}\Psi$ is a $p$-torsion module. Thus we deduce that $\Psi$ is injective. Therefore we may view $\op{Sel}_{\op{Gr}}(A/\Q_\infty)^\vee$ as a $\Lambda$-submodule of $\op{Sel}_{\op{Gr}}(A'/\Q_\infty)^\vee$. We set $X:=\op{Sel}_{\op{Gr}}(A'/\Q_\infty)^\vee$ and $Y:=\op{Sel}_{\op{Gr}}(A/\Q_\infty)^\vee$. We have shown that $Y$ is contained in $X$ and $X/Y$ has exponent $p$. We also note that by Proposition \ref{propn no nontriv} that $X$ has no non-trivial finite $\Lambda$-submodules. Therefore by Proposition \ref{g(X) >= g(Y)/2}, we conclude that $g(X)\geq g(Y)/2$, i.e., $g(A')\geq g(A)/2$.  
\end{proof}

 \begin{proposition}\label{g(A')>=n/2}
    Suppose that the following conditions hold:
\begin{enumerate}
    \item $H^0(\Q, A)=0$ and $H^0(\Q, A')=0$,
    \item there is a prime $\ell$ for which $H^0(\Q_\ell, A)$ is finite,
    \item $\mu(A)=0$.
\end{enumerate}   
    Moreover assume that there are $n$ primes $\ell_1, \dots, \ell_n\in S$ such that $\delta_{\ell_i}(A)>0$ for all $i=1, \dots, n$. Then, we find that $g(A)\geq n$ and $g(A')\geq n/2$.
 \end{proposition}
\begin{proof}
    Proposition \ref{g(A)>= sum delta l} asserts that
    \[g(A)\geq \sum_{\ell\in S} \delta_\ell(A)\geq n.\]
    Lemma \ref{g(A')>=g(A)/2} then implies that $g(A')\geq g(A)/2\geq n/2$.
\end{proof}

 \begin{proposition}\label{prop BK lower bound}
Assume that:\begin{enumerate}
    \item $H^0(\Q, A)=0$ and $H^0(\Q, A')=0$,
    \item there is a prime $\ell$ for which $H^0(\Q_\ell, A)$ is finite,
    \item $\mu(A)=0$.
\end{enumerate} Moreover, assume that $\delta_{\ell_i}(A)>0$ for $i=1, \dots, n$ for a set of primes $\ell_1, \dots, \ell_n\in S$. Then, 
\[\op{dim}_{\F_p} \op{Sel}_{\op{BK}}(A'/\Q)[p]\geq n/2-\sum_{\ell\in S} \delta_\ell(A').\]
 \end{proposition}
 \begin{proof}
     It follows from Proposition \ref{g(A')>=n/2} that $g(A')\geq n/2$.  
     We have an exact sequence
     \[0\rightarrow \op{Sel}_{\op{BK}}(A'/\Q)[p]\rightarrow \op{Sel}_{\op{Gr}}(A'/\Q_\infty)^\Gamma [p]\rightarrow \bigoplus_{\ell\in S} \op{ker}h_{\ell,A'}[p].\]
     Thus we have that 
     \[\op{dim}_{\F_p} \op{Sel}_{\op{BK}}(A'/\Q)[p]\geq g(A')-\sum_{\ell\in S} \delta_\ell(A'),\] and the result follows. 
 \end{proof}

 \section{Deformations of reducible Galois representations}\label{s 4}
The deformation theory of Galois representations is a powerful tool introduced by Mazur. Over the years, it has been employed crucially in establishing the modularity of Galois representations $\rho$ in characteristic zero, when it is known that the residual representation $\bar{\rho}$ is modular. We shall use this technique to construct modular Galois representations with prescribed local properties. Starting with a residual representation $\bar{\rho}:\op{G}_{\Q}\rightarrow \op{GL}_2(\F_q)$, we study the functor of deformations $\rho_R:\op{G}_{\Q}\rightarrow \op{GL}_2(R)$ of $\bar{\rho}$, where $R$ is a local ring with residue field $\F_q$. We begin by setting up some notation and recalling some well-known facts from Galois cohomology.

\subsection{Cohomological facts} Given a finite dimensional vector space over $\F_p$, we set $V^\vee:=\op{Hom}\left(V, \F_p\right)$. Let $M$ be a finite $\F_p[\op{G}_{\Q}]$-module, denote by $M^*:=\op{Hom}\left(M, \mu_p\right)$ its Tate dual. For $i\in \{1, 2\}$ and a finite set of primes $\Sigma$ containing $p$ and the primes which ramify for the Galois action on $M$, set 
\[\Sh_{\Sigma}^i(M):=\op{ker}\left(H^i(\Q_\Sigma/\Q, M)\longrightarrow \bigoplus_{\ell\in \Sigma} H^i(\Q_\ell, M)\right).\] By the Poitou-Tate duality theorem 
\[\Sh_{\Sigma}^2(M)\simeq \Sh_{\Sigma}^1(M^*)^\vee.\]

For each prime $\ell\in \Sigma\cup \{\infty\}$, let $\cL_\ell$ be a subspace of $H^1(\Q_\ell, M)$. Since $p$ is odd, $H^1(\mathbb{R}, M)=0$, and in particular, this forces $\cL_\infty=0$. The data $\cL=\{\cL_\ell\}_{\ell\in \Sigma}$ is known as a \emph{Selmer datum}, and the associated Selmer group is defined as the following kernel:
\[H^1_{\cL}(\op{G}_{\Sigma}, M):=\op{ker}\left\{H^1(\Q_{\Sigma}/\Q, M)\longrightarrow \bigoplus_{\ell\in \Sigma} \left(\frac{H^1(\Q_\ell, M)}{\cL_\ell}\right)\right\}.\]
Let $\cL_\ell^\perp\subseteq H^1(\Q_\ell, M^*)$ consist of classes that pair to $0$ with $\cL_\ell$. The dual Selmer group is defined as follows:
\[H^1_{\cL^\perp}(\op{G}_{\Sigma}, M^*):=\op{ker}\left\{H^1(\Q_{\Sigma}/\Q, M^*)\longrightarrow \bigoplus_{\ell\in \Sigma} \left(\frac{H^1(\Q_\ell, M^*)}{\cL_\ell^\perp}\right)\right\}.\]
Wiles' formula \cite[Theorem 8.7.9]{NSW} gives a relationship between the dimensions of the Selmer and dual Selmer groups:
\begin{equation}\label{wiles formula}\begin{split}\op{dim} H^1_{\cL}(\op{G}_{\Sigma}, M)-\op{dim} H^1_{\cL^\perp}(\op{G}_{\Sigma}, M^*) = & \op{dim} H^0(\Q, M)-\op{dim} H^0(\Q, M^*) \\ &+ \sum_{\ell\in \Sigma\cup \{\infty\}} \left(\op{dim}\cL_\ell-\op{dim} H^0(\Q_\ell, M) \right).\end{split}\end{equation}
 The Poitou--Tate sequence gives the following exact sequence:
\begin{equation}\label{PT les}
    \begin{split}
        0 &\rightarrow  H^1_{\cL}(\op{G}_{\Sigma}, M)\rightarrow H^1(\Q_{\Sigma}/\Q, M)\rightarrow \bigoplus_{\ell\in \Sigma} \left(\frac{H^1(\Q_\ell, M)}{\cL_\ell}\right) \\
        & \rightarrow H^1_{\cL^\perp}(\op{G}_{\Sigma}, M^*)^\vee\rightarrow H^2(\Q_\Sigma/\Q, M)\rightarrow \bigoplus_{\ell\in \Sigma} H^2(\Q_\ell, M),
    \end{split}
\end{equation}
see for instance, \cite[p.~555, l.7]{tayloricosahedral}.
\subsection{Deformation functors}
\par A significant part of the theory is understanding how deformations behave locally at each prime $\ell$. There are local-global principles that can be formulated via Galois cohomological techniques. In order to understand global deformations of $\bar{\rho}$, it is enough to understand the local deformations at each prime $\ell$ and ensure that these local deformations satisfy certain conditions encapsulated by subfunctors. Let $\cO := W(\F_q)$ denote the ring of Witt vectors with residue field $\F_q$ and note that $p$ is the uniformizer of $\cO$. Consider a character $\varphi: \op{G}_{\Q} \rightarrow \F_q^\times$ and the reducible Galois representation

\[
\bar{\rho} = \begin{pmatrix} \varphi & * \\ & 1 \end{pmatrix}: \op{G}_{\Q} \rightarrow \op{GL}_2(\F_q).
\]

The isomorphism class of the entry $*$ can be viewed as a cohomology class in $H^1(\Q, \F_q(\varphi))$. Let $S$ be the set of primes $\ell$ such that either $\ell = p$ or $\ell \neq p$ and $\bar{\rho}$ is ramified at $\ell$. Given a prime $\ell$ and a Galois representation $\rho: \op{G}_{\Q} \rightarrow \op{GL}_2(\cdot)$, we denote by $\rho_{|\ell}$ the restriction of $\rho$ to $\op{G}_{\ell}$.

Let $\op{CNL}_{\cO}$ be the category whose objects are complete noetherian local $\cO$-algebras $R$ with maximal ideal $\mathfrak{m}$, such that $R/\mathfrak{m} \cong \F_q$ as an $\cO$-algebra. Given $R \in \op{CNL}_{\cO}$, there is a unique $\cO$-algebra map $\pi: R \rightarrow \F_q$, which induces the isomorphism $R/\mathfrak{m} \xrightarrow{\sim} \F_q$. A morphism in this category is a map of $\cO$-algebras $\phi: R_1\rightarrow R_2$. An important family of examples includes $\cO/p^m$ for $m \in \Z_{\geq 1}$.

Let $\cG$ denote the group $\op{G}_{\Q}$ or $\op{G}_{\ell}$ (for a prime $\ell$), and let $\bar{\varrho}$ denote $\bar{\rho}$ or $\bar{\rho}_{|\ell}$ respectively. A \textit{lift} $\varrho: \cG \rightarrow \op{GL}_2(R)$ of $\bar{\varrho}$ is a homomorphism that reduces to $\bar{\varrho}$, as illustrated in the following commutative diagram:

\[
\begin{tikzpicture}[node distance = 2.5 cm, auto]
    \node at (0,0) (G) {$\cG$};
    \node (A) at (3,0) {$\op{GL}_2(\F_q)$.};
    \node (B) at (3,2) {$\op{GL}_2(R)$};
    \draw[->] (G) to node [swap]{$\bar{\varrho}$} (A);
    \draw[->] (B) to node{} (A);
    \draw[->] (G) to node {$\varrho$} (B);
\end{tikzpicture}
\]
The vertical arrow is induced by the quotient map $\pi: R\rightarrow \F_q$. The set of lifts $\varrho:\cG\rightarrow \op{GL}_2(R)$ is denoted $\op{Lift}_{\bar{\varrho}}(R)$. This association gives rise to a functor of lifts
\[\op{Lift}_{\bar{\varrho}}: \op{CNL}_{\cO}\rightarrow \op{Sets}.\]
Given a prime $\ell$, and $\rho\in \op{Lift}_{\bar{\rho}}(R)$, the restriction $\rho_{|\ell}$ is a lift of $\bar{\rho}_{|\ell}$. This gives rise to a natural transformation
\[\op{Lift}_{\bar{\rho}}\rightarrow \op{Lift}_{\bar{\rho}_{|\ell}}.\]

\begin{definition}
We set $\widehat{\op{GL}_2}(R)$ to be the kernel of the reduction map 
\[\op{GL}_2(R)\rightarrow \op{GL}_2(\F_q).\]Two lifts $\varrho, \varrho' : \cG \rightarrow  \op{GL}_2(R)$ of $\bar{\varrho}$ are said to be \textit{strictly equivalent} if $\varrho' = A\varrho A^{-1}$ for some $A \in \widehat{\op{GL}_2}(R)$. A \textit{deformation} is a strict equivalence class of lifts.
\end{definition}
Denote by $\op{Def}_{\bar{\rho}}$ and $\op{Def}_{\bar{\rho}_{|\ell}}$ the associated functors of deformations. For each prime $\ell$, the restriction of a global deformation to $\op{G}_{\ell}$ gives a natural transformation 
\[\op{Def}_{\bar{\rho}}\rightarrow \op{Def}_{\bar{\rho}_{|\ell}}.\]
\begin{definition}
    Let $\ell$ be a prime. A \emph{deformation functor} at $\ell$ is a subfunctor $\mathcal{C}_\ell$ of $\op{Def}_{\bar{\rho}_{|\ell}}$. For each coefficient ring $R$, the set $\mathcal{C}_\ell(R)$ is a subset of $\op{Def}_{\bar{\rho}_{|\ell}}(R)$. We say that a deformation $\varrho : \op{G}_\ell \to \op{GL}_2(R)$ of $\bar{\rho}_{\restriction \ell}$ satisfies $\mathcal{C}_\ell$ if $\varrho \in \mathcal{C}_\ell(R)$. 
\end{definition} 
\subsection{Tangent and obstruction spaces}
The lifting strategy is based on a local-global principle, where local deformation functors are fixed and global deformations are required to satisfy these functors at the primes where they are allowed to ramify.

\begin{definition}\label{adgaloisaction}
Let $\op{Ad}\bar{\varrho}$ denote the Galois module whose underlying vector space consists of $2 \times 2$ matrices with entries in $\mathbb{F}_q$. The Galois action is as follows: for $g \in \cG$ and $v \in \op{Ad}\bar{\varrho}$, set $g \cdot v := \bar{\rho}(g) v \bar{\rho}(g)^{-1}$. Let $\op{Ad}^0\bar{\varrho}$ be the $\cG$-stable submodule of trace zero matrices.  
\end{definition}
We note that $\op{Ad}\bar{\varrho}$ breaks up into a direct sum
\[\op{Ad}\bar{\varrho}=\op{Ad}^0\bar{\varrho}\oplus \F_q\cdot \op{Id},\] where the second summand is generated by the scalar matrices. Let $D\in \mathcal{C}_{\cO}$ be the ring of dual numbers given by $D:=\mathbb{F}_q[\epsilon]/(\epsilon^2)$. An infinitesimal deformation $\varrho: \cG\rightarrow \op{GL}_2(D)$ of $\bar{\varrho}$ is a deformation of $\bar{\varrho}$ to $\op{GL}_2(D)$. We describe how infinitesimal deformations are in bijection with cohomology classes $f\in H^1(\cG, \op{Ad}\bar{\rho})$. Let $\widetilde{f}\in Z^1(\cG, \op{Ad}\bar{\rho})$ be a cocycle representing $f$. We check that the formula 
\[(\op{Id}+\epsilon \widetilde{f})\bar{\rho}: \cG\rightarrow\op{GL}_2(D)\] defines a homomorphism lifting $\bar{\rho}$. We see that 
\[\begin{split}& (\op{Id}+\epsilon \widetilde{f}(\sigma_1))\bar{\rho}(\sigma_1)(\op{Id}+\epsilon \widetilde{f}(\sigma_2))\bar{\rho}(\sigma_2) \\
& =(\op{Id}+\epsilon \widetilde{f}(\sigma_1)) (\op{Id}+\epsilon \bar{\rho}(\sigma_1)\widetilde{f}(\sigma_2)\bar{\rho}(\sigma_1)^{-1})\bar{\rho}(\sigma_1)\bar{\rho}(\sigma_2) \\
& = (\op{Id}+\epsilon \widetilde{f}(\sigma_1)+\epsilon \sigma_1\widetilde{f}(\sigma_2))\bar{\rho}(\sigma_1\sigma_2) \\ 
& = (\op{Id}+\epsilon \widetilde{f}(\sigma_1\sigma_2))\bar{\rho}(\sigma_1\sigma_2),\end{split}\] and it is clear from our definition that $(\op{Id}+\epsilon \widetilde{f})\bar{\rho}$ reduces to $\bar{\rho}$ modulo $(\epsilon)$. The deformation represented by $(\op{Id}+\epsilon \widetilde{f})\bar{\rho}$ depends only on $f$, and is referred to as the \emph{twist} by $f$. We denote this deformation by $(\op{Id}+\epsilon f)\bar{\rho}$. The association 
\[f\mapsto (\op{Id}+\epsilon f)\bar{\rho}\] gives a bijection 
\[H^1(\cG, \op{Ad}\bar{\rho})\leftrightarrow \op{Def}_{\bar{\varrho}}(D).\]
The mod $(\epsilon)$ reduction map $D\rightarrow\F_q $ is a special case of what's known as a \emph{small extension}, and we can extend the above bijection to a generalized setting. 

\begin{definition}
    Let $R$ and $S$ be coefficient rings over $\cO$ with maximal ideals $\mathfrak{m}_R$ and $\mathfrak{m}_S$ respectively. A surjective map of $\cO$-algebras $\pi: R\rightarrow S$ is said to be a \emph{small extension} if:
    \begin{itemize}
        \item $\op{ker}\pi$ is generated by a single element $t$,
        \item $t \mathfrak{m}_R=0$. 
    \end{itemize}
\end{definition}
\begin{remark}We note that for $n\geq 1$, the reduction map
\[\pi_n : \cO/(p^{n+1})\rightarrow \cO/(p^n)\] is a small extension, with $\op{ker}\pi_n$ generated by $p^n$. 
\end{remark}
Given a small extension $\pi: R\rightarrow S$ with kernel $(t)$, consider the reduction map at the level of deformations
\[\pi^*: \op{Def}_{\bar{\varrho}}(R)\rightarrow \op{Def}_{\bar{\varrho}}(S)\] defined by $\pi^*(\varrho):=\varrho\mod{(t)}$. 
Suppose that $\varrho_S\in \op{Def}_{\bar{\varrho}}(S)$ is a deformation of $\bar{\varrho}$ such that there exists $\varrho_R\in \op{Def}_{\bar{\varrho}}(R)$ such that $\varrho_S=\varrho_R\mod{(t)}$. Let $f\in H^1(\cG, \op{Ad}\bar{\rho})$ and set 
\[\varrho_R':=(\op{Id}+t f)\varrho_R.\] Here, the product $t f$ is well defined since $t\mathfrak{m}_R=0$. Then $\varrho_R'$ is a well defined homomorphism satisfying the property that $\pi^*(\varrho_R')=\varrho_S$. We refer to $\varrho_R'$ as the \emph{twist} of $\varrho_R$ by $f$. This twisting operation \[H^1(\cG, \op{Ad}\bar{\rho})\times (\pi^*)^{-1}(\varrho_S)\rightarrow (\pi^*)^{-1}(\varrho_S)\] is simply transitive and $(\pi^*)^{-1}(\varrho_S)$ is an $H^1(\cG, \op{Ad}\bar{\rho})$-torsor (assuming that it is non-empty).

\par Let $\nu: \cG\rightarrow \cO^\times$ be a choice of a lift of $\op{det} \bar{\varrho}$. Given $R\in \op{CNL}_{\cO}$, set $\nu_R:\cG\rightarrow R^\times$ to denote the composite of $\nu$ with the structure map $\cO^\times \rightarrow R^\times$. We take $\op{Def}_{\bar{\varrho}}^\nu(R)$ to be the subset of $\op{Def}_{\bar{\varrho}}(R)$ consisting of deformations $\varrho$ such that $\op{det}\varrho=\nu_R$. Now suppose that $\varrho_R\in \op{Def}_{\nu}^\nu(R)$ and $f\in H^1(\cG, \op{Ad}^0\bar{\rho})$. Then the twist $\varrho_R':=(\op{Id}+t f) \varrho_R$ also has the property that $\op{det}\varrho_R'=\nu_R$. Consider the map 
\[\pi_\nu^*: \op{Def}_{\bar{\varrho}}^\nu(R)\rightarrow \op{Def}_{\bar{\varrho}}^\nu(S) \]
induced from restricting $\pi^*$ to deformations to those with determinant $\nu$. Then, we find that $(\pi_\nu^*)^{-1}(\varrho_S)$ is an $H^1(\cG, \op{Ad}^0\bar{\rho})$-torsor.

\par We now discuss the cohomological obstructions to lifting. Let $\pi: R\rightarrow S$ be a small extension and $\varrho_S\in \op{Def}_{\bar{\varrho}}(S)$. Let $A=\mtx{a}{b}{c}{-a}\in \op{Ad}\bar{\rho}$ and choose a lift $\widetilde{A}=\mtx{\widetilde{a}}{\widetilde{b}}{\widetilde{c}}{-\widetilde{a}}$ with entries in $R$. The entries $\widetilde{a}, \widetilde{b}, \widetilde{c}\in R$ are lifts of $a, b, c$ respectively. Then, we set $t A$ to denote the matrix $t \widetilde{A}:=\mtx{t \widetilde{a}}{t\widetilde{b}}{t\widetilde{c}}{-t\widetilde{a}}$. Since $t\mathfrak{m}_R=0$, it is easy to see that $t\widetilde{A}$ is independent of the choice of $\widetilde{A}$. Note that $t^2=0$, and therefore the determinant of $\op{Id}+t A$ is seen to be equal to $1$. Setting 
\[H:=\op{ker}\left(\op{SL}_2(R)\rightarrow \op{SL}_2(S)\right),\] we find that the association taking $A\in \op{Ad}^0\bar{\rho}$ to $\op{Id}+t A$ gives an isomorphism 
\[\op{Ad}^0\bar{\rho}\xrightarrow{\sim} H.\]
Let $\varrho:\cG\rightarrow \op{GL}_2(S)$ be a deformation of $\bar{\varrho}$ with determinant $\nu_S$. Note that there always exists a set theoretic lift \[\widetilde{\varrho}: \cG\rightarrow \op{GL}_2(R)\] with determinant $\nu_R$. Note that  $\widetilde{\varrho}(\sigma_1\sigma_2)\widetilde{\varrho}(\sigma_2)^{-1}\widetilde{\varrho}(\sigma_1)^{-1}$ is trivial modulo $(t)$ and has determinant equal to $1$. Thus, we can write \[\widetilde{\varrho}(\sigma_1\sigma_2)\widetilde{\varrho}(\sigma_2)^{-1}\widetilde{\varrho}(\sigma_1)^{-1}=\op{Id}+t \cO(\sigma_1, \sigma_2).\] The map $\cO(\varrho):\cG\times \cG\rightarrow \op{Ad}^0\bar{\rho} $ is a $2$-cocycle, we call $\cO(\varrho)\in H^2(\cG, \op{Ad}^0\bar{\rho})$ the \emph{obstruction to lifting $\varrho$}. Thus, $\cO(\varrho)=0$ if and only if a homomorphism $\widetilde{\varrho}$ lifting $\varrho$ does exist.

\subsection{The lifting result of Hamblen and Ramakrishna}
 \par In this subsection, we prove a variant of \cite[Theorem 2]{hamblenramakrishna}. Given a finite set of primes $\Sigma$, recall that $\Q_\Sigma\subset \bar{\Q}$ is the maximal extension in which primes $\ell\notin \Sigma$ are unramified, and set $\op{G}_\Sigma:=\op{Gal}(\Q_\Sigma/\Q)$. Denote by $\chi$ the $p$-adic cyclotomic character and $\bar{\chi}$ its mod $p$ reduction. Given a character $\varphi:\op{G}_{\Q}\rightarrow \F_q^\times$, let $\widetilde{\varphi}$ denote the Teichm\"uller lift of $\varphi$. We take $\omega$ to denote the Teichm\"uller character, i.e., the Teichm\"uller lift of $\bar{\chi}$. Although the result applies to both even and odd representations, we only concern ourselves with the statement in the odd case. 

\begin{theorem}[Hamblen and Ramakrishna]\label{hamblen ramakrishna thm}
    Let $p$ be an odd prime number and $q$ be a power of $p$. Set $\F_q$ to be the field with $q$ elements and let $\cO$ denote its ring of Witt vectors with fraction field $\cK$. Consider a reducible mod $p$ Galois representation 
    \[\bar{\rho}=\mtx{\varphi}{\ast}{0}{1}:\op{G}_{\Q}\rightarrow \op{GL}_2(\F_q)\] and $\Sigma$ be the set of primes containing $p$ and the primes at which $\bar{\rho}$ is ramified. Assume that the following conditions hold:
    \begin{enumerate}
        \item $\bar{\rho}$ is indecomposable, 
        \item $\varphi^2\neq 1$ and $\varphi\notin \{ \bar{\chi}, \bar{\chi}^{-1}\}$. Moreover, the $\F_p$-span of the image of $\varphi$ is all of $\F_q$; 
        \item $\bar{\rho}_{|p}$ is not an unramified unipotent representation $\mtx{1}{\ast}{0}{1}$. 
    \end{enumerate}
    Then there a finite set of primes $X$ containing $\Sigma$ and a deformation $\rho: \op{G}_{X}\rightarrow \op{GL}_2(\cO)$ such that 
    \begin{itemize}
        \item $\rho$ is irreducible, 
        \item $\op{det}\rho=\widetilde{\varphi} \chi^{p^2(p-1)}$.
        \item The local representation $\rho_{|p}$ is ordinary, i.e., there is an unramified character $\gamma: \op{G}_p\rightarrow \cO^\times$ such that 
        \[\rho_{|p}=\mtx{\widetilde{\varphi} \chi^{p^2(p-1)} \gamma}{\ast}{0}{\gamma^{-1}}.\]
    \end{itemize}
\end{theorem}
The lift $\rho$ is \emph{geometric} in the sense of Fontaine and Mazur \cite{fontainemazur}. A result of Skinner and Wiles \cite{skinnerwiles} shows that $\rho$ is modular, i.e., arises from a stable lattice in a Galois representation associated to a modular form. We remark that Theorem \ref{hamblen ramakrishna thm} has been generalized to reducible Galois representations taking values in higher rank symplectic groups \cite{Raydefreducible}, and by Fakhruddin, Khare and Patrikis \cite{FKP} for general reductive groups (the latter work also relaxes some conditions on $\bar{\rho}$ in the $2$-dimensional case).

\par Let us fix a lift $\nu: \op{G}_{\Q} \rightarrow \cO^\times$ of the determinant character $\varphi = \det \bar{\rho}$ of the form $\nu=\alpha \chi^{k-1}$ where $k\geq 2$ is an integer and $\alpha$ is a character of prime to $p$ order. Throughout the remainder of Section \ref{s 4}, all deformations under consideration will have determinant $\nu$. Recall that Hamblen and Ramakrishna chose the specific lift $\nu := \widetilde{\varphi} \chi^{p^2(p-1)}$; we extend their method to allow for more general lifts $\nu$ satisfying the aforementioned condition on the determinant. 

\par The strategy employed by Ramakrishna \cite{RaviInventMath, RamakrishnaJRMS} and Hamblen--Ramakrishna \cite{hamblenramakrishna} is based on constructing a compatible family of deformations $\rho_n$ of $\bar{\rho}$ as depicted:
\[\begin{tikzpicture}[node distance = 1.5 cm, auto]
      \node (GSX) at (0,0){$\op{G}_{\Q}$};
      \node (GL2) at (5,0){$\op{GL}_{2}(\cO/p^{n-1}).$};
      \node (GL2Wn) at (3,2)[above of= GL2]{$\op{GL}_{2}(\cO/p^{n})$};
      \node (GL2Wnplus1) at (5,4){$\op{GL}_{2}(\cO/p^{n+1})$};
      \draw[->] (GSX) to node [swap]{$\rho_{n-1}$} (GL2);
      \draw[->] (GL2Wn) to node {} (GL2);
      \draw[->] (GSX) to node [swap]{$\rho_n$} (GL2Wn);
      \draw[->] (GL2Wnplus1) to node {} (GL2Wn);
      \draw[dashed,->] (GSX) to node {$\rho_{n+1}$} (GL2Wnplus1);
      \end{tikzpicture}\] 
      In the diagram above, $\rho_n$ may not admit a lift to a mod $p^{n+1}$ deformation $\rho_{n+1}$. However, upon replacing $\rho_n$ with $(\op{Id}+p^n f) \rho_n$ for a suitable cohomology class $f\in H^1(\op{G}_{\Q}, \op{Ad}\bar{\rho})$ it is shown that it can be lifted to $\rho_{n+1}$. This construction is repeated at every step. Then $\rho:\op{G}_{\Q}\rightarrow\op{GL}_2(\cO)$ is taken to be the inverse limit $\varprojlim_n \rho_n$. The construction of suitable cohomology classes $f$ (at every step) is based on a local-global argument, where certain local deformation conditions are prescribed at finitely many primes, and $\rho_n$ is constructed so as to satisfy these local conditions defined by subfunctors $\cC_\ell\subseteq \op{Def}_{\bar{\rho}_{|\ell}}$. Thus by definition, for any coefficient ring $R\in \op{CNL}_{\cO}$, we have that $\cC_\ell(R)$ is a subset of $\op{Def}_{\bar{\rho}_{|\ell}}(R)$. A deformation $\varrho\in \op{Def}_{\bar{\rho}_{|\ell}}(R)$ is said to satisfy $\cC_\ell$ if it is contained in $\cC_\ell(R)$. The condition $\cC_\ell$ is \emph{representable} if there is a coefficient ring $R_\ell^{\op{univ}}\in \op{CNL}_{\cO}$ and a universal deformation 
 \[\varrho_\ell^{\op{univ}}: \op{G}_\ell\rightarrow \op{GL}_2(R_\ell^{\op{univ}})\] 
 such that any deformation $\varrho\in \cC_\ell(R)$ arises from a unique map $h: R_\ell^{\op{univ}}\rightarrow R$ as a composition 
 \[\op{G}_\ell\xrightarrow{\varrho_\ell^{\op{univ}}} \op{GL}_2(R_\ell^{\op{univ}})\rightarrow \op{GL}_2(R), \] where the second map is induced by applying $h$ to the entries of $\op{GL}_2(R_\ell^{\op{univ}})$.

 \begin{theorem}[Grothendieck]
    Let $\cC_\ell$ be a subset of $\op{Def}_{\bar{\rho}_{|\ell}}$ defined by a deformation condition that satisfies the following properties:
    \begin{enumerate}
        \item $\mathcal{C}_\ell(\mathbb{F}_q)$ contains only $\bar{\rho}_{\restriction \ell}$.
        \item For $i=1,2$, suppose $R_i \in \mathcal{C}_{\mathcal{O}}$ and $\rho_i \in \mathcal{C}_\ell(R_i)$. If $I_1$ is an ideal of $R_1$ and $I_2$ an ideal of $R_2$ such that there is an isomorphism $\alpha: R_1/I_1 \xrightarrow{\sim} R_2/I_2$ with $\alpha(\rho_1 \mod I_1) = \rho_2 \mod I_2$, then consider the fiber product \[R_3 = \{(r_1, r_2) \mid \alpha(r_1 \mod I_1) = r_2 \mod I_2\}\] and the induced $R_3$-representation $\rho_1 \times_{\alpha} \rho_2$. It follows that $\rho_1 \times_{\alpha} \rho_2 \in \mathcal{C}_\ell(R_3)$.
        \item If $R \in \op{CNL}_{\mathcal{O}}$ with maximal ideal $\mathfrak{m}_R$, and $\rho \in \op{Def}_{\bar{\rho}_{|\ell}}(R)$ such that $\rho \in \mathcal{C}_\ell(R/\mathfrak{m}_R^n)$ for all $n > 0$, then $\rho \in \mathcal{C}_\ell(R)$. This means the functor $\mathcal{C}_\ell$ is continuous.
    \end{enumerate}
Under these conditions, the functor $\cC_\ell$ is representable. 
 \end{theorem}

 \begin{definition}\label{def functor defn}
     Let $\cC_\ell\subseteq \op{Def}_{\bar{\rho}_{|\ell}}$ be a subfunctor. We say that $\cC_\ell$ is \emph{liftable} if for every small extension $\pi: R\rightarrow S$, the induced map 
     \[\cC_\ell(R)\rightarrow \cC_\ell(S)\] is surjective.
 \end{definition}

\begin{definition}\label{preserves}
    Let $\ell$ be a prime and $(\cC_\ell, \cN_\ell)$ be a pair, where $\cC_\ell$ is a functor of deformations of $\bar{\rho}_{|\ell}$ in the sense of Definition \ref{def functor defn} and $\cN_\ell$ is a subspace of $H^1(\op{G}_\ell, \op{Ad}^0\bar{\rho})$. We say that $\cN_\ell$ preserves $\cC_\ell$ if the following condition holds. Let $\pi: R\rightarrow S$ be a small extension and $t$ be a generator of $\op{ker}\pi$. Let $\varrho\in \cC_\ell(R)$ and $f\in \mathcal{N}_\ell$. Then, the twist
$\varrho':=(\op{Id}+f t) \varrho$ is also contained in $\cC_\ell(R)$. Note that both $\varrho$ and $\varrho'$ reduce to the same deformation in $\cC_\ell(S)$.
\end{definition}
For $\ell\in \Sigma \cup \{\infty\}$, Ramakrishna constructed pairs $(\cC_\ell, \cN_\ell)$ (see \cite{RaviInventMath} and \cite{RamakrishnaFM}). These pairs have the following properties:
\begin{enumerate}
    \item $\cC_\ell$ is liftable in the sense of Definition \ref{def functor defn},
    \item $\cC_p$ consists of ordinary deformations. These consist of local deformations $\varrho: \op{G}_{p}\rightarrow \op{GL}_2(R)$ given by 
\[\varrho=\mtx{\nu\gamma}{\ast}{0}{\gamma^{-1}},\] where $\gamma: \op{G}_p\rightarrow R^\times$ is an unramified character. 
    \item $\cN_\ell$ preserves $\cC_\ell$ in the sense of Definition \ref{preserves}, 
    \item \begin{equation}\label{equation for dimension of Nell}\op{dim} \cN_\ell=\begin{cases}
        \dim H^0(\op{G}_\ell, \op{Ad}^0\bar{\rho})+1 &\text{ if }\ell=p;\\
        \dim H^0(\op{G}_\ell, \op{Ad}^0\bar{\rho}) &\text{ if }\ell\neq p, \infty;\\
        0 &\text{ if }\ell=\infty.
    \end{cases}\end{equation}
\end{enumerate}

This setting is identical to that of Hamblen and Ramakrishna, cf. \cite[Fact 5]{hamblenramakrishna}. We refer to \cite[section 4]{PatrikisExceptional} and \cite[p.~551]{tayloricosahedral} for a detailed exposition on the construction of these pairs, and their properties. In addition to these primes, there are certain auxiliary primes called trivial primes which are allowed to ramify. We recall the definition from \cite[section 5]{hamblenramakrishna} here.
\begin{definition}
    A prime $\ell$ is said to be a \emph{trivial prime} if the following conditions are satisfied:
    \begin{enumerate}
        \item $\ell\equiv 1\mod{p}$ and $\ell\not \equiv 1\mod{p^2}$, 
        \item the restriction of $\bar{\rho}$ to $\op{G}_\ell$ is trivial.
    \end{enumerate}
\end{definition}

\begin{lemma}\label{triv primes positive density}
    The set of trivial primes has positive density.
\end{lemma}
\begin{proof}
    We note that a prime $\ell$ is a trivial prime if $\ell$ splits in $\Q(\mu_{p}, \bar{\rho})$ and does not split in $\Q(\mu_{p^2})$. We show that $\Q(\mu_{p^2})$ is not contained in $\Q(\mu_{p}, \bar{\rho})$ and the result is then a consequence of the Chebotarev density theorem. Note that $\Q(\mu_{p^2})$ contains $\Q_1$, the first layer in the cyclotomic $\Z_p$-extension of $\Q$.
    
    \par Suppose by way of contradiction that $\Q(\mu_{p^2})$ is contained in $\Q(\mu_{p}, \bar{\rho})$. Then, it follows that $\Q_1$ is contained in $\Q(\mu_{p}, \bar{\rho})$. We show first that $\Q_1$ is contained in $\Q(\bar{\rho})$. The extension $\Q(\mu_{p}, \bar{\rho})/\Q(\bar{\rho})$ is a prime to $p$ extension, and $\Q_1\cdot \Q(\bar{\rho})$ is an extension of $\Q(\bar{\rho})$ of degree dividing $p$. Since $\Q_1$ is contained in $\Q(\mu_{p}, \bar{\rho})$, it follows that $\Q_1$ is contained in $\Q(\bar{\rho})$. We note that $\bar{\rho}$ is an indecomposable representation of the form $\bar{\rho}=\mtx{\varphi}{\ast}{}{1}$. Let $G$ be the image of $\bar{\rho}$. The map $\sigma\mapsto \bar{\rho}(\sigma)$ identifies $\op{Gal}(\Q(\bar{\rho})/\Q)$ with $G$. We write $G=T\ltimes N$ where $T$ (resp. $N$) consists of diagonal (resp. unipotent) matrices in $G$. Note that $T$ is a prime to $p$ group and $N$ is the Sylow $p$-subgroup. Since $\varphi$ is a non-trivial character, the semidirect product is not a direct product and $T$ is not a normal subgroup. Therefore, $G=\op{Gal}(\Q(\bar{\rho})/\Q)$ does not have any quotient isomorphic to $\Z/p\Z$. This implies that $\Q_1$ is not contained in $\Q(\bar{\rho})$, a contradiction. Hence, we have shown that $\Q(\mu_{p^2})$ is not contained in $\Q(\mu_p, \bar{\rho})$ and from this, the result follows.
\end{proof}
Hamblen and Ramakrishna define two possible choices of $(\cC_\ell, \mathcal{N}_\ell)$ for any trivial prime $\ell$, where $\cC_\ell$ is a functor of deformations of $\bar{\rho}_{|\ell}$ in the sense of Definition \ref{def functor defn} and $\cN_\ell$ is a subspace of $H^1(\op{G}_\ell, \op{Ad}^0\bar{\rho})$. Let $\ell$ be a trivial prime. Note that since $\bar{\rho}_{|\ell}$ is trivial, it follows that any deformation $\varrho\in \op{Def}_{\bar{\rho}_{|\ell}}(R)$ takes values in $\widehat{\op{GL}}_2(R)$, which is a pro-$p$ group. In particular, the image of $\op{I}_\ell$ is a pro-$p$ group. Since $\ell\neq p$, it follows that $\varrho$ is tamely ramified. Let $\sigma_\ell$ be a Frobenius at $\ell$ and denote by $\tau_\ell$ a generator of the pro-$p$ quotient of the tame inertia group of $\Q_\ell$. The pro-$p$ completion of $\op{G}_\ell$ is generated by $\sigma_\ell$ and $\tau_\ell$ and is subject to the single relation $\sigma_\ell \tau_\ell \sigma_\ell^{-1}=\tau_\ell^\ell$. Thus, prescribing a deformation $\varrho$ is equivalent to choosing two matrices 
\[\varrho(\sigma_\ell), \varrho(\tau_\ell)\in \widehat{\op{GL}}_2(R)\] subject to the relation 
\[\varrho(\sigma_\ell) \varrho(\tau_\ell) \varrho(\sigma_\ell)^{-1}=\varrho(\tau_\ell)^\ell.\]
Choose a square root of $\ell$ in $\Z_p$ which satisfies the congruence $\ell^{1/2}\equiv 1\mod{p}$. For $R\in \mathcal{C}_{\cO}$, let $\mathcal{D}_\ell(R)\subset \op{Def}_{\bar{\rho}_{|\ell}}(R)$ consist of deformations with a representative $\varrho:\op{G}_{\ell}\rightarrow \op{GL}_2(R)$ satisfying:
   \begin{equation}
       \varrho(\sigma_\ell)=\ell^{\frac{k-2}{2}}\mtx{\ell}{x}{0}{1}\text{ and } \varrho(\tau_\ell)=\mtx{1}{y}{0}{1},
   \end{equation}
   for $x,y\in R$. Since $\bar{\varrho}$ is the trivial representation of $\op{G}_\ell$, we find that $x$ and $y$ are elements in the maximal ideal of $R$. Hamblen and Ramakrishna define classes in $H^1(\op{G}_{\ell}, \op{Ad}^0\bar{\rho})$ as follows:
\[
\begin{split}
f_1(\sigma_\ell)=\mtx{0}{1}{0}{0} \text{, } &f_1(\tau_\ell)=\mtx{0}{0}{0}{0},\\ 
f_2(\sigma_\ell)=\mtx{0}{0}{0}{0} \text{, } & f_2(\tau_\ell)=\mtx{0}{1}{0}{0},\\  g^{\operatorname{nr}}(\sigma_\ell)=\mtx{0}{0}{1}{0} \text{, } & g^{\operatorname{nr}}(\tau_\ell)=\mtx{0}{0}{0}{0},\\  g^{\operatorname{ram}}(\sigma_\ell)=\mtx{0}{0}{1}{0} \text{, } & g^{\operatorname{ram}}(\tau_\ell)=\mtx{-\frac{y}{\ell-1}}{0}{0}{\frac{y}{\ell-1}}.
\end{split}\]
\par The space $\mathcal{Q}_\ell$ denotes the subspace of $H^1(\op{G}_{\ell}, \op{Ad}^0\bar{\rho})$ spanned by $\{f_1,f_2\}$, let $\mathcal{P}_\ell^{\operatorname{nr}}$ the subspace spanned by $\{f_1,f_2,g^{\operatorname{nr}}\}$ and $\mathcal{P}_\ell^{\operatorname{ram}}$ the subspace spanned by $\{f_1,f_2,g^{\operatorname{ram}}\}$. The functor $\mathcal{D}_\ell^{\op{nr}}$ (resp. $\mathcal{D}_\ell^{\op{ram}}$) was introduced in \cite[Proposition 24]{hamblenramakrishna} (resp. \cite[Proposition 28]{hamblenramakrishna}). The deformation functor $\mathcal{D}_\ell^{\op{nr}}$ consists of deformations satisfying $\mathcal{D}_\ell$ for which $p^2|x,y$. On the other hand, $\mathcal{D}_\ell^{\op{ram}}$ consists of deformations satisfying $\mathcal{D}_\ell$ for which $p^2|x$ and $p||y$. 
\begin{definition}\label{trivial local def} We introduce two different deformation conditions for $\rho_{|\ell}$ for a trivial prime $\ell$:
    \begin{description}
        \item[Type I] $(\cC_\ell, \mathcal{N}_\ell)$ is the conjugate of $(\mathcal{D}_\ell^{\op{nr}}, \mathcal{P}_\ell^{\op{nr}})$ by $\mtx{1}{0}{1}{1}$.
        \item[Type II] $(\cC_\ell, \mathcal{N}_\ell)$ is the conjugate of $(\mathcal{D}_\ell^{\op{ram}}, \mathcal{P}_\ell^{\op{ram}})$ by $\mtx{0}{1}{1}{0}$.
    \end{description} 
\end{definition}\noindent Note here that if $A\in \op{GL}_2(\cO)$ is an invertible matrix, and $\varrho$ is a deformation of $\bar{\rho}_{\restriction \ell}$, then so is the conjugate $A \varrho A^{-1}$. This is because $\bar{\rho}_{\restriction \ell}$ is the trivial representation. Mod $p^2$ deformations of $\bar{\rho}_{|\ell}$ are unramified (resp. ramified) if they are of Type \rm{I} (resp. Type \rm{II}) since $p^2|y$ (resp. $p||y$).

\begin{proposition}\label{Nl preserves Cl for mod p^3}
    If $\ell$ is a trivial prime, then $\mathcal{N}_\ell$ preserves $\cC_\ell$ for the small extension $\cO/p^n\rightarrow \cO/p^{n-1}$ for $n\geq 3$. 
\end{proposition}
\begin{proof}
    \cite[Corollary 25 and Corollary 29]{hamblenramakrishna}.
\end{proof}

\par The lifting argument consists of 2 steps.
\begin{itemize}
    \item First, it is shown that there is a finite set of primes $X$ containing $\Sigma$, such that $\bar{\rho}$ lifts to a representation $\rho_3: \op{G}_{X}\rightarrow \op{GL}_2(\cO/p^3)$. Here $X\setminus \Sigma$ consists only of trivial primes. Moreover, $\rho_3$ is shown to satisfy certain local conditions $\cC_\ell$ for all primes $\ell\in X$. For each prime $\ell\in \Sigma$, the pair $(\cC_\ell, \mathcal{N}_\ell)$ is the one fixed earlier in this section satisfying properties (1)--(4) above. For the primes $\ell\in X\setminus \Sigma$, the pair $(\cC_\ell, \mathcal{N}_\ell)$ is one of the two choices from Definition \ref{trivial local def}. 
    \item The representation $\rho_3$ is shown to lift to a compatible family
    \[\rho_n: \op{G}_X\rightarrow \op{GL}_2(\cO/p^n),\] satisfying $\cC_\ell$ for all primes $\ell\in X$. Given $\rho_n$ satisfying these conditions for $n\geq 3$, Hamblen and Ramakrishna show that a lift $\rho_{n+1}$ exists. This argument relies on showing that the natural map 
    \[H^1(\op{G}_X, \op{Ad}^0\bar{\rho})\longrightarrow \bigoplus_{\ell\in X} \left(\frac{H^1(\op{G}_\ell, \op{Ad}^0\bar{\rho})}{\mathcal{N}_\ell}\right)\] is an isomorphism. The kernel and cokernel of this map are Selmer and dual Selmer groups, which are shown to vanish for an appropriate choice of primes $X$. The formulation of this part of the argument is due to Taylor \cite{tayloricosahedral}. 
\end{itemize}

Let us now describe these steps in further detail. Let $U_1, U_2, U_3$ be the submodules of $\op{Ad}^0\bar{\rho}$ consisting of matrices $\mtx{}{b}{}{}$, $\mtx{a}{b}{}{-a}$ and $\mtx{a}{b}{c}{-a}$ respectively. Note that $U_3=\op{Ad}^0\bar{\rho}$. Hamblen and Ramakrishna \cite[Proposition 13]{hamblenramakrishna} show that there is a finite set of primes $T$ containing $\Sigma$ such that 
\begin{itemize}
    \item $T\backslash \Sigma$ consists only of trivial primes,
    \item $\Sh_T^2(M)=0$ for $M=U_i$ for $i=1,2,3$.
\end{itemize}
The obstruction to lifting $\bar{\rho}$ to a reducible representation $\xi_2:\op{G}_{T}\rightarrow \op{GL}_2(\cO/p^2)$ lies in $\Sh_T^2(U_1)$. This group is $0$, hence a reducible deformation $\xi_2$ exists. It is then shown that there is a set of trivial primes $Z$ disjoint from $T$ and a global cohomology class $z\in H^1(\op{G}_{Z\cup T}, \op{Ad}^0\bar{\rho})$ such that the twist $\rho_2:=(\op{Id}+pz)\xi_2$ satisfies the conditions $\cC_\ell$ for all primes $\ell\in Z\cup T$. For $\ell\in T\setminus \Sigma$, $\cC_\ell$ is Type I and for $\ell\in Z$, $\cC_\ell$ is Type II. This step involves the method of Khare, Larsen and Ramakrishna \cite{KLR}. Set $X_1:=Z\cup T$. Let \[P_2:=\op{ker}\left\{ \op{SL}_2(\cO/p^2)\rightarrow \op{SL}_2(\cO/p)\right\}\] be the principal congruence subgroup of $\op{SL}_2(\cO/p^2)$ and $D_2$ consist of scalar matrices
\[D_2:=\left\{ \mtx{1+pd}{0}{0}{1+pd}\in \op{GL}_2(\cO/p^2)\mid d\in \cO/p\right\}.\] Let $B_2\subset \op{GL}_2(\cO/p^2)$ be the Borel subgroup of upper triangular matrices $\mtx{\ast}{\ast}{0}{\ast}$.

\begin{lemma}\label{new lemma 1}
    With respect to notation above, let $G_2\subseteq \op{GL}_2(\Z/p^2\Z)$ be the image of $\rho_2$. The following assertions hold:
    \begin{enumerate}
        \item $D_2\cdot G_2$ contains $P_2$.
        \item The representation $\rho_2$ is irreducible, i.e., $G_2$ is not contained in a conjugate of $B_2$.
    \end{enumerate}
\end{lemma}
\begin{proof}
    The proof is similar to \cite[Proposition 42]{hamblenramakrishna}, with the key difference arising from our convention that $\op{det}\rho_2$ need not be the trivial character mod $p^2$. Let $\mathcal{K}$ denote the Galois group $\op{Gal}(\Q(\rho_2)/\Q(\bar{\rho}))$. For $\sigma\in \mathcal{K}$, we have $\bar{\rho}(\sigma)=\op{Id}$ and we write $\rho_2(\sigma)=\op{Id}+p\mtx{a'}{b'}{c'}{d'}$ where $\mtx{a'}{b'}{c'}{d'}\in \op{Ad}\bar{\rho}$. Set $\iota(\sigma):=\mtx{a'}{b'}{c'}{d'}$ so that $\rho_2(\sigma)=\op{Id}+p\iota(\sigma)$. We thus have an injection $\iota : \mathcal{K}\hookrightarrow \op{Ad}\bar{\rho}$. Let $\mathcal{N}\subseteq \op{Ad}\bar{\rho}$ be the image of $\iota$. Note that $\op{Ad}\bar{\rho}$ decomposes as: \[\op{Ad}\bar{\rho}=\F_q\cdot \op{Id}\oplus \op{Ad}^0\bar{\rho},\]
    and this involves writing 
    \[\mtx{a'}{b'}{c'}{d'}=d\op{Id}+\mtx{a}{b}{c}{-a},\] where $a=\frac{a'-d'}{2}$, $b=b'$, $c=c'$ and $d=\frac{a'+d'}{2}$.
    Letting $\pi: \op{Ad}\bar{\rho}\rightarrow \op{Ad}^0\bar{\rho}$ denote the projection map, we claim that $\pi(\mathcal{N})=\op{Ad}^0\bar{\rho}$.
    
    \par Assuming the claim, we prove part (1). Given $A\in P_2$, we show that there exists $D\in D_2$ and $\sigma\in \mathcal{K}$ such that $A=D\rho_2(\sigma)$. We write $A=\op{Id}+p\mtx{a}{b}{c}{-a}$. Since $\pi(\mathcal{N})=\op{Ad}^0\bar{\rho}$, we find that there exists $\sigma\in \mathcal{K}$ such that $\pi_2(\iota(\sigma))=\mtx{a}{b}{c}{-a}$. Writing $\iota(\sigma)=d\op{Id}+\mtx{a}{b}{c}{-a}$, we find that
    \[\rho_2(\sigma)=\op{Id}+p\left(d\op{Id}+\mtx{a}{b}{c}{-a}\right)=\left(\op{Id}+p\mtx{d}{}{}{d}\right)\left(\op{Id}+p\mtx{a}{b}{c}{-a}\right).\]
    Thus we have that $A=D\rho_2(\sigma)$ where $D=\op{Id}+p\mtx{-d}{}{}{-d}$. We deduce that $P_2\subseteq D_2\cdot G_2$.
    \par We prove the claim. Let $\ell\in Z$, recall that in this case $\rho_{2|\ell}$ satisfies $\cC_\ell$ of Type II. Then, we have that $\rho_2(\tau_\ell)=\mtx{1}{0}{y}{1}$ where $p||y$. Write $y=py'$ where $y'\in \cO/p$ is non-zero. Note that $\iota(\tau_\ell)=\mtx{0}{0}{y'}{0}$, thus, $\mathcal{N}$ contains $\mtx{0}{0}{y'}{0}$. This implies that $\pi(\mathcal{N})$ contains $\mtx{0}{0}{y'}{0}$. Any proper Galois submodule of $\op{Ad}^0\bar{\rho}$ consists of upper triangular matrices $\mtx{a}{b}{0}{-a}$ and thus $\pi(\mathcal{N})$ is all of $\op{Ad}^0\bar{\rho}$ which proves the claim. 
    \par Next we show that (2) follows from (1). Suppose by way of contradiction that there exists $g\in \op{GL}_2(\cO/p^2)$ such that $G_2\subseteq gB_2g^{-1}$. Note that $D_2$ consists of scalar matrices, we have that $g D_2 g^{-1}=D_2$ and thus $D_2\cdot G_2\subseteq g B_2 g^{-1}$. From part (1), we deduce that $P_2$ is contained in $g B_2 g^{-1}$ and therefore, $g^{-1} P_2 g \subseteq B_2$. Since $P_2=g^{-1} P_2 g$ we find that $P_2\subseteq B_2$, which is a contradiction. Therefore, we have shown that $\rho_2$ is irreducible.
\end{proof}
We remark that since $\rho_2$ is irreducible, any characteristic zero lift $\rho$ of $\rho_2$ will be irreducible.

\par Consider the obstruction class $\mathcal{O}(\rho_2)\in H^2(\op{G}_{X_1}, \op{Ad}^0\bar{\rho})$ for lifting $\rho_2$ to a representation $\xi_3: \op{G}_{X_1}\rightarrow \op{GL}_2(\cO/p^3)$. Since the deformation conditions $\cC_\ell$ are all liftable, this obstruction class lies in $\Sh^2_{X_1}(\op{Ad}^0\bar{\rho})$. Recall that $\Sh^2_T(\op{Ad}^0\bar{\rho})=0$. Since $X_1$ contains $T$, we find that $\Sh^2_{X_1}(\op{Ad}^0\bar{\rho})=0$ as well, and thus in particular, $\mathcal{O}(\rho_2)=0$. Thus, $\rho_2$ lifts to $\xi_3:\op{G}_{X_1}\rightarrow \op{GL}_2(\cO/p^3)$. 
\begin{proposition}\label{inf many triv primes f psi conditions}
    Let $Y$ be a finite set of primes containing $\Sigma$ and suppose that there exist non-zero classes $f\in H^1(\op{G}_Y, \op{Ad}^0\bar{\rho})$ and $\psi\in H^1(\op{G}_Y, \op{Ad}^0\bar{\rho}^*)$. Let $\beta: \op{G}_Y\rightarrow \F_q^\times$ be a character. Then there is an infinite set of trivial primes $\ell\notin Y$ such that 
\begin{enumerate}
    \item $\rho_{2|\ell}$ satisfies the Type I condition $\cC_
    \ell$,
    \item $\ell\equiv 1+2p \pmod{p^2}$,
    \item $f(\sigma_\ell)\neq \mtx{-a}{a}{b}{a}$ for any values of $a,b\in \F_q$, 
    \item $\gamma_{|\ell}=0$ for all $\gamma\in H^1(\op{G}_Y, U_i^*)$ for $i=1,2$,
    \item $\psi_{|\ell}\neq 0$ and $\psi$ can be extended to a basis of $H^1(\op{G}_Y, \op{Ad}^0\bar{\rho}^*)$ such that $\alpha_{|\ell}=0$ for all other elements $\alpha$ of the basis,
    \item $\beta_{|\ell}=1$.
\end{enumerate}
\end{proposition}
\begin{proof}
    The argument extends that of \cite[Proposition 44]{hamblenramakrishna}. In fact, we prove that the set of primes $\ell$ satisfying the conditions above has positive density and is defined by a non-empty Chebotarev condition. More precisely, there is a Galois number field $\cF/\Q$ and a non-empty subset $\mathcal{S}$ of $\op{Gal}(\cF/\Q)$ such that if $\ell$ is a prime which is unramified in $\cF$ and $\sigma_\ell\in \mathcal{S}$, then $\ell$ is a trivial prime which satisfies conditions (1)--(6). Set $\cF_0:=\Q(\bar{\rho}, \mu_{p^2})$, from the proof of Lemma \ref{triv primes positive density}, there is a non-empty set $\cS_0\subseteq \op{Gal}(\cF_0/\Q)$ such that if $\ell$ is a prime which is unramified in $\cF_0$ and $\sigma_\ell\in \cS_0$, then $\ell$ is a trivial prime. In fact, $\cS_0$ consists of the non-trivial elements in the subgroup $\op{Gal}(\cF_0/\Q(\bar{\rho}, \mu_p))$. Each of the conditions ($i$) for $i=1, \dots, 6$ is determined by a non-empty Chebotarev condition, i.e., there is a Galois number field $\cF_i/\Q$ and a non-empty $\mathcal{S}_i\subset \op{Gal}(\cF_i/\Q)$ such that if $\ell$ is unramified in $\cF_i$ and $\sigma_\ell\in \cS_i$ then $\ell$ satisfies condition ($i$). Setting $\cF:=\cF_0\cdots \cF_6$, let $\cS$ consist of all $\sigma\in \op{Gal}(\cF/\Q)$ such that $\sigma_{|\cF_i}\in \cS_i$ for $i=0,\dots, 6$. The conditions (1)--(6) are shown to be \emph{compatible}, i.e., the set $\cS$ is shown to be non-empty. By the Chebotarev density theorem, the set of trivial primes $\ell$ satisfying (1)--(6) is infinite with density at least $\frac{|\mathcal{S}|}{[\cF:\Q]}$.

    \par The conditions (2)--(5) coincide with the conditions (2)--(5) of \cite[Proposition 44]{hamblenramakrishna}. The fields $\cF_i$ are listed there for $i=2, \dots, 5$. Condition (6) is the condition that $\ell$ splits completely in $\Q(\beta)$. For condition (1), note that $\rho_2$ is unramified at $\ell$ since $\ell\notin Y$. Let $\mathcal{K}:=\op{Gal}\left(\Q(\rho_2)/\Q(\bar{\rho})\right)$ and consider the injection $\iota: \mathcal{K}\hookrightarrow \op{Ad}\bar{\rho}$ from the proof of Lemma \ref{new lemma 1}. This gives a surjection $\op{Gal}(\Q(\rho_2)/\Q(\bar{\rho}))\twoheadrightarrow \op{Ad}^0\bar{\rho}$ and thus gives rise to a subfield $\Q(\rho_2)'$ of $\Q(\rho_2)$ with $\op{Gal}(\Q(\rho_2)'/\Q(\bar{\rho}))\simeq \op{Ad}^0\bar{\rho}$. Let us describe this isomorphism in greater detail. Given $\sigma\in \mathcal{K}$, write \[\rho_2(\sigma)=\op{Id}+p\left(\mtx{d}{}{}{d}+\mtx{a}{b}{c}{-a}\right)\] so $\op{det}\rho_2(\sigma)=1+2dp \pmod{p^2}$. Then $\sigma$ is mapped to $\mtx{a}{b}{c}{-a}\in \op{Ad}^0\bar{\rho}$. This association factors through $\op{Gal}(\Q(\rho_2)'/\Q(\bar{\rho}))$. Let $\ell$ be a trivial prime such that $\ell\equiv 1+2p \pmod{p^2}$. Note that $\op{det}\rho_2(\sigma_\ell)=\alpha(\sigma_\ell) \chi^{k-1}(\sigma_\ell)=\ell^{k-1}\equiv 1+2(k-1)p\pmod{p^2}$ since $\alpha(\sigma_\ell)=\varphi(\sigma_\ell)=1$.
    Hence $d=k-1\pmod{p}$ and 
    \[\rho_2(\sigma_\ell)=\op{Id}+p\left(\mtx{k-1}{}{}{k-1}+\mtx{a}{b}{c}{-a}\right).\] In order for $\rho_2$ to satisfy $\cC_\ell$ of Type I, it is necessary that 
    \[\rho_2(\sigma_\ell)=\ell^{\frac{k-2}{2}}\mtx{\ell-x}{x}{\ell-x-1}{x+1}\text{ and } \rho_2(\tau_\ell)=\mtx{1-y}{y}{-y}{1+y}.\] Since $p^2|x,y$, we must have that 
    \[\rho_2(\sigma_\ell)=\ell^{\frac{k-2}{2}}\mtx{\ell}{}{\ell-1}{1}\text{ and } \rho_2(\tau_\ell)=\mtx{1}{}{}{1}.\]
    Without loss of generality, we may assume the second condition holds since it holds for all but finitely many primes $\ell$. As for the first condition, since $\ell\equiv 1+2p\pmod{p^2}$, we must have
    \[\rho_2(\sigma_\ell)=\op{Id}+p\mtx{k}{0}{2}{k-2}.\] Thus, taking $\mtx{a}{b}{c}{-a}=\mtx{1}{0}{2}{-1}$, we get an element $\sigma\in \op{Gal}(\Q(\rho_2)'/\Q)$ such that if $\sigma_\ell=\sigma$, then condition (1) holds. Thus we set $\cF_1:=\Q(\rho_2)'$ and $\cS_1:=\{\sigma\}$. It remains to show that the conditions (1)--(6) are compatible, i.e., the set $\mathcal{S}\in \op{Gal}(\cF/\Q)$ is indeed non-empty. The arguments of \cite[Proposition 44]{hamblenramakrishna} show verbatim that (1)--(5) can be simultaneously satisfied, with $\Q(\rho_2)'$ replacing $\Q(\rho_2)$. This is to say that if we set $\cF':=\cF_0\dots \cF_5$, there is a non-empty set $\cS'\subset \op{Gal}(\cF'/\Q(\bar{\rho}))$ such that if $\ell$ is unramified in $\cF'$ and $\sigma_\ell\in \cS'$ then $\ell$ is a trivial prime which satisfies (1)--(5). On the other hand, condition (6) is a splitting condition in a prime to $p$ extension $\Q(\beta)/\Q$ and is easily seen to be independent of (1)--(5). Note that $\cF:=\cF_0\dots \cF_6=\cF'\cdot \Q(\beta)$. Let $\cS\subset \op{Gal}(\cF/\Q)$ consist of $\sigma$ such that $\sigma_{|\Q(\beta)}=1$ and $\sigma_{|\cF'}\in \cS'$, the set $\cS$ is non-empty and thus the set of trivial primes $\ell$ satisfying (1)--(6) has positive density. In particular, this set of primes is infinite.
\end{proof}

We remark that in this section only conditions (1)--(5) are used, it is in the proof of Proposition \ref{prop bounding dim of Selmer} that condition (6) is essential.

\par Recall that $X_1:=T\cup Z$. For $\ell \in T\setminus \Sigma$, $\cN_\ell$ is Type \rm{I}. For $\ell\in Z$, recall that $\cN_\ell$ is Type \rm{II}. Suppose that $H^1_{\cN}(\op{G}_{X_1}, \op{Ad}^0\bar{\rho})\neq 0$ and $H^1_{\cN^\perp}(\op{G}_{X_1}, \op{Ad}^0\bar{\rho}^*)\neq 0$, we choose non-zero classes $f\in H^1_{\cN}(\op{G}_{X_1}, \op{Ad}^0\bar{\rho})$ and $\psi\in H^1_{\cN^\perp}(\op{G}_{X_1}, \op{Ad}^0\bar{\rho}^*)$. It follows from Proposition \ref{inf many triv primes f psi conditions} and \cite[Proposition 46]{hamblenramakrishna} that there is a finite set of primes $X$ containing $X_1$ such that the Selmer group $H^1_{\mathcal{N}}(\op{G}_X, \op{Ad}^0\bar{\rho})$ and the dual Selmer group $H^1_{\mathcal{N}^\perp}(\op{G}_X, \op{Ad}^0\bar{\rho}^*)$ are both $0$. The primes $\ell\in X\setminus X_1$ are trivial primes for which $(\mathcal{C}_\ell, \mathcal{N}_\ell)$ are chosen to be of Type I. This step is called \emph{killing the dual Selmer group}.

The vanishing of these Selmer groups and the long exact sequence \eqref{PT les} implies that there is an isomorphism:
\begin{equation}\label{global selmer isomorphism 1}H^1(\op{G}_X, \op{Ad}^0\bar{\rho})\xrightarrow{\sim} \bigoplus_{\ell\in X} \left(\frac{H^1(\Q_\ell, \op{Ad}^0\bar{\rho})}{\mathcal{N}_\ell}\right).\end{equation}
For each prime $\ell\in X$, pick $f_\ell\in H^1(\Q_\ell, \op{Ad}^0\bar{\rho})$ such that 
\begin{equation}\label{local eqn}(\op{Id}+p^2 f_\ell) (\xi_3)_{|\ell}\in \cC_\ell.\end{equation} Proposition \ref{Nl preserves Cl for mod p^3} asserts that $\cN_\ell$ preserves $\cC_\ell$ for small extensions $\cO/p^n\rightarrow \cO/p^{n-1}$ for $n\geq 3$. It thus follows that any translate of $f_\ell$ by an element in $\mathcal{N}_\ell$ will also satisfy \eqref{local eqn}. Thus, we can view $f_\ell$ as an element in the quotient $\left(\frac{H^1(\Q_\ell, \op{Ad}^0\bar{\rho})}{\mathcal{N}_\ell}\right)$. The tuple \[(f_\ell)\in \bigoplus_{\ell\in X} \left(\frac{H^1(\Q_\ell, \op{Ad}^0\bar{\rho})}{\mathcal{N}_\ell}\right)\]arises from a global cohomology class $f\in H^1(\op{G}_X, \op{Ad}^0\bar{\rho})$ since the map \eqref{global selmer isomorphism 1} is an isomorphism. We set $\rho_3:=(\op{Id}+p^2f)\xi_3$ and find that $\rho_3$ satisfies the liftable deformation conditions $\mathcal{C}_\ell$ at all primes $\ell\in X$. Note that this argument does not apply to show $\rho_2$ satisfies the $\cC_\ell$ conditions for all $\ell\in X$ since it is not known that $\cN_\ell$ preserves $\cC_\ell$ for the extension $\cO/p^2\rightarrow \cO/p$. This is why \cite{KLR} is necessary when constructing $\rho_2$.

\par Since $\rho_3$ satisfies the liftable deformation conditions $\cC_\ell$ for $\ell\in X$, the obstruction to lifting $\rho_3$ to 
\[\xi_4: \op{G}_X\rightarrow \op{GL}_2(\cO/p^4)\] lies in $\Sh_X^2(\op{Ad}^0\bar{\rho})$. This group is evidently $0$ since $X$ contains $T$. Hence, the lift $\xi_4:\op{G}_X\rightarrow \op{GL}_2(\cO/p^4)$ exists. There is a tuple 
\[(g_\ell)\in \bigoplus_{\ell\in X} \left(\frac{H^1(\Q_\ell, \op{Ad}^0\bar{\rho})}{\mathcal{N}_\ell}\right)\] which is defined by the property that for all $\ell\in X$, 
\[(\op{Id}+p^3 g_\ell) (\xi_4)_{|\ell}\in \cC_\ell.\] This tuple arises from a global cohomology class $g\in H^1(\op{G}_X, \op{Ad}^0\bar{\rho})$ since the map \eqref{global selmer isomorphism 1} is an isomorphism. We take $\rho_4:=(\op{Id}+p^3 g)\xi_4$ and find that $\rho_4$ by design satisfies the conditions $\cC_\ell$ at all primes $\ell\in X$. In this way, we continue the argument inductively to obtain a compatible family of deformations 
\[\rho_n: \op{G}_X\rightarrow \op{GL}_2(\cO/p^n)\] which satisfy the conditions $\cC_\ell$ for all primes $\ell\in X$. Take $\rho:\op{G}_X\rightarrow \op{GL}_2(\cO)$ to denote the inverse limit $\varprojlim_n \rho_n$. We find that $\rho$ is odd, irreducible and satisfies the conditions $\cC_\ell$ at the primes $\ell\in X$. In particular, it satisfies $\cC_p$, which makes it ordinary at $p$. If $\varphi_{|p}\neq 1$, then a result of Skinner and Wiles \cite[\emph{Theorem} on p. 6]{skinnerwiles} implies that $\rho$ is modular. We remark that recent work of Pan \cite{Pan} removes the condition $\varphi_{|p}\neq 1$, though for us later it will always hold.

\section{Main results}\label{s 5}

\par In this section, we prove the main result of this article, i.e., Theorem \ref{main thm}. Let $p\geq 5$. We start with a mod $p$ representation 
\[\bar{\rho}=\mtx{\varphi}{\ast}{}{1}:\op{G}_{\Q}\rightarrow \op{GL}_2(\F_p).\] Assume that:
\begin{enumerate}
    \item $\bar{\rho}$ is indecomposable,
    \item $\varphi=\bar{\chi} \beta^2$ where $\beta:\op{G}_{\Q}\rightarrow \F_p^\times$ is an odd character which is unramified at $p$, 
    \item $\beta^2\neq 1$.
\end{enumerate}
\noindent It is clear that $\bar{\rho}$ satisfies the hypotheses of Theorem \ref{hamblen ramakrishna thm} with $q = p$. In greater detail, the assumption $\beta^2 \neq 1$ implies that $\varphi \neq \bar{\chi}$. Since $\beta$ is unramified at $p$, we have $\varphi_{|\op{I}_p} = \bar{\chi}_{|\op{I}_p}$. As $p \geq 5$, it follows that $\bar{\chi}^2_{|\op{I}_p}\neq 1$. Thus $\varphi^2 \neq 1$ and $\varphi \neq \bar{\chi}^{-1}$. Finally, since $\varphi_{|p} \neq 1$, one finds that $\bar{\rho}_{|p}$ is not of the form $\mtx{1}{\ast}{}{1}$.
\par Recall that $\Sigma$ is the set of primes containing $p$ and the primes at which $\bar{\rho}$ is ramified. Let $\nu$ be the character $\chi\widetilde{\beta}^2$. We extend the arguments of Hamblen and Ramakrishna to show that modular lifts of $\bar{\rho}$ exist which also satisfy a number of local conditions at a prescribed set of primes. These local conditions will guarantee large $p$-rank in the Bloch--Kato Selmer group associated to a certain Galois stable lattice. Suppose that $\rho:\op{G}_{\Q}\rightarrow \op{GL}_2(\Z_p)$ is a $p$-ordinary modular lift of $\bar{\rho}$ with determinant $\nu$ and let $f$ be the associated eigencuspform. Denote by $T_\rho$ the $\Z_p$-lattice corresponding to $\rho$ and set $V_\rho:=T_\rho\otimes_{\Z_p} \Q_p$. We set $\psi:=\widetilde{\beta}^{-1}$. Note that $\psi^2=\chi\nu^{-1}=\widetilde{\beta}^{-2}$.
\par Observe that $k=2$ is the weight of $f$ and $\widetilde{\beta}^2$ is its nebentype. Let $V$ be a self-dual twist of $V_\rho$. More specifically, $V=V_\rho(\psi)$. Likewise, we set $T:=T_\rho(\psi)$ and $A:=V/T$. Write $T= \Z_p e_1 \oplus \Z_p e_2$, which we recall to be the underlying Galois module of $\rho\otimes \psi$. We set $e_1':=p^{-1} e_1$ and $e_2':=e_2$, and consider the lattice $T':=\Z_p e_1'\oplus \Z_p e_2'$ contained in $V$. We let 
\[(\rho\otimes \psi)':\op{G}_{\Q}\rightarrow \op{Aut}_{\Z_p}(T')\xrightarrow{\sim} \op{GL}_2(\Z_p)\] be the Galois representation associated to $T'$ with respect to the basis $(e_1', e_2')$. Since $\bar{\rho}\otimes \bar{\psi}$ is upper triangular, we can write $(\rho\otimes \psi)(g)=\mtx{a}{b}{pc}{d}$, where $a, b, c, d \in \Z_p$. One finds that 
\[(\rho\otimes \psi)'(g)=\mtx{p}{0}{0}{1}\mtx{a}{b}{pc}{d}\mtx{p}{0}{0}{1}^{-1}=\mtx{a}{pb}{c}{d}.\]
In particular, the mod $p$ reduction of $(\rho\otimes \psi)'$ is of the form $\mtx{\bar{\psi}\varphi}{}{\ast}{\bar{\psi}}$.

We set \[A:=V/T=(\Q_p/\Z_p)e_1\oplus (\Q_p/\Z_p)e_2\text{ and }A':=V/T'=(\Q_p/\Z_p)e_1'\oplus (\Q_p/\Z_p)e_2',\]
and recall that $\op{Sel}_{\op{BK}}(A/\Q)$ and $\op{Sel}_{\op{BK}}(A'/\Q)$ are the associated Bloch--Kato Selmer groups.

We now describe conditions at a prime $\ell$ for $\delta_\ell(A)>0$ and $\delta_\ell(A')=0$ to be simultaneously satisfied. Let $\ell$ be a trivial prime and consider the deformations $(\cC_\ell, \mathcal{N}_\ell)$ of Type \rm{II} from the previous section. Recall that $(\cC_\ell, \mathcal{N}_\ell)$ is the conjugate of $(\mathcal{D}_\ell^{\op{ram}}, \mathcal{P}_\ell^{\op{ram}})$ by $\mtx{0}{1}{1}{0}$. Thus, $\cC_\ell$ are deformations $\varrho:\op{G}_\ell\rightarrow \op{GL}_2(\Z_p/p^n)$ of $\bar{\rho}_{|\ell}$ with a representative satisfying 
 \begin{equation}
\varrho(\sigma_\ell)=\mtx{1}{0}{x}{\ell}\text{ and } \varrho(\tau_\ell)=\mtx{1}{0}{y}{1},
   \end{equation}
   where $p^2|x$ and $p||y$. The action on $A$, with respect to the $\Q_p/\Z_p$-basis $(e_1, e_2)$, is via $\varrho\otimes \psi$, with representative given by:
   \begin{equation}\label{rho times psi eq 1}
\varrho\otimes \psi(\sigma_\ell)=\psi(\sigma_\ell)\mtx{1}{0}{x}{\ell}\text{ and } \varrho\otimes \psi(\tau_\ell)=\psi(\tau_\ell)\mtx{1}{0}{y}{1},
   \end{equation}
where $p^2|x$ and $p||y$. On the other hand, the action on $A'$, with respect to the $\Q_p/\Z_p$-basis $(e_1', e_2')$, is via $(\varrho\otimes \psi)'$, given by:
   \begin{equation}\label{rho times psi eq 2}
(\varrho\otimes \psi)'(\sigma_\ell)=\psi(\sigma_\ell)\mtx{1}{0}{x/p}{\ell}\text{ and } (\varrho\otimes \psi)'(\tau_\ell)=\psi(\tau_\ell)\mtx{1}{0}{y/p}{1},
   \end{equation}
 where $p^2|x$ and $p||y$. Let $\ell$ be a trivial prime. Since $\bar{\rho}_{|\ell}$ is trivial, it follows that $\varphi_{|\ell}=1$. Recall that $\varphi=\bar{\chi} \beta^2$ and $\bar{\chi}_{|\ell}=1$ since $\ell\equiv 1 \pmod{p}$. Thus we find that $\beta_{|\ell}^2=1$. If $\beta$ is unramified at $\ell$, then $\beta(\sigma_\ell)=\pm 1$. 
\begin{lemma}\label{lemma type II local}
    Let $\ell$ be a trivial prime. Assume that $\rho_{|\ell}$ satisfies $\cC_\ell$ of Type \rm{II}, then, we have that $\delta_\ell(A')=0$. Furthermore, if $\beta_{|\ell}=1$, then $\delta_\ell(A)=1$.
\end{lemma}
\begin{proof}
    \par First we show that $\delta_\ell(A')=0$. Let $\mathcal{A}':=(A')^{\op{I}_\ell}/(A')^{\op{I}_\ell}_{\op{div}}$, and recall from Lemma \ref{formula for the Tamagawa number} that
    \[\delta_\ell(A')=\dim_{\F_p}\left(\left(\mathcal{A}'/(\sigma_\ell-1)\mathcal{A}'\right)[p]\right).\]
    We show that $\mathcal{A}'=0$, i.e., that $(A')^{\op{I}_\ell}$ is divisible. From this we deduce that $\delta_\ell(A')=0$. Recall from \eqref{rho times psi eq 2} that $\tau_\ell$ acts on $A'=\Q_p/\Z_p e_1'\oplus \Q_p/\Z_p e_2'$ by the matrix $\psi(\tau_\ell)\mtx{1}{0}{y/p}{1}$, where $p||y$. Recall that $\psi=\widetilde{\beta}^{-1}$ and thus $\psi(\tau_\ell)=1$. Thus, the action of $\tau_\ell$ on $\mathcal{A}'$ is via $\mtx{1}{0}{y/p}{1}$. Since $y/p$ is a unit in $\Z_p$, it follows that $(A')^{\op{I}_\ell}=\Q_p/\Z_p e_2'$. In particular, $(A')^{\op{I}_\ell}$ is divisible and thus, $\mathcal{A}'=0$. We have thus proven that $\delta_\ell(A')=0$.
\par Now we assume that $\beta_{|\ell}=1$ and we show that $\delta_\ell(A)=1$. Since $\psi=\widetilde{\beta}^{-1}$, we have that $\psi_{|\ell}=1$, and in particular, $\psi(\sigma_\ell)=1$ and $\psi(\tau_\ell)=1$. Let $\mathcal{A}:=A^{\op{I}_\ell}/(A^{\op{I}_\ell})_{\op{div}}$, and recall from Lemma \ref{formula for the Tamagawa number} that
    \[\delta_\ell(A)=\dim_{\F_p}\left(\left(\mathcal{A}/(\sigma_\ell-1)\mathcal{A}\right)[p]\right).\] By \eqref{rho times psi eq 1}, $\tau_\ell$ acts on $A$ via the matrix $\mtx{1}{0}{y}{1}$.
    Since $p||y$, we thus find that as a $\Z_p$-module
    \begin{equation}\label{AIleqn}A^{\op{I}_\ell}=\left(p^{-1}\Z_p/\Z_p \right)e_1 \oplus \Q_p/\Z_p e_2.\end{equation}
    Thus, $(A^{\op{I}_\ell})_{\op{div}}=\Q_p/\Z_p e_2$ and $\mathcal{A}=\left(p^{-1}\Z_p/\Z_p \right)e_1$. On the other hand,
    \[(\sigma_\ell-1) e_1= \left(\psi(\sigma_\ell)-1\right)e_1+\psi(\sigma_\ell)xe_2=xe_2.\]
     We thus write 
    \[(\sigma_\ell-1)(p^{-1} e_1)=\left(\frac{x}{p}\right)e_2, \] and find that $(\sigma_\ell-1)(p^{-1} e_1)\in (\Z_p/\Z_p) e_1\oplus (\Z_p/\Z_p) e_2$. Thus, we have shown that
     $(\sigma_\ell-1) \mathcal{A}=0$. Hence, $\mathcal{A}/(\sigma_\ell-1)\mathcal{A}=\mathcal{A}$ and we conclude that $\delta_\ell(A)=1$. 
\end{proof}

\begin{definition}\label{defn of type III} We say that $(\cC_\ell, \cN_\ell)$ is of Type \rm{III} if it is the conjugate of $(\mathcal{D}_\ell^{\op{ram}}, \mathcal{P}_\ell^{\op{ram}})$ by $\mtx{0}{1}{1}{1}$.
\end{definition}
  The action of $\sigma_\ell$ and $\tau_\ell$ on $A$ with respect to the basis $(e_1, e_2)$ is given by: \begin{equation}\label{rho times psi eq 3}
\varrho\otimes \psi(\sigma_\ell)=\psi(\sigma_\ell)\mtx{1}{0}{(x+1-\ell)}{\ell}\text{ and } \varrho\otimes \psi(\tau_\ell)=\psi(\tau_\ell)\mtx{1}{0}{y}{1}.
   \end{equation}
   The action on $A'$ on the other hand is given by:
   \begin{equation}\label{rho times psi eq 4}
\varrho\otimes \psi(\sigma_\ell)=\psi(\sigma_\ell)\mtx{1}{0}{(x+1-\ell)/p}{\ell}\text{ and } \varrho\otimes \psi(\tau_\ell)=\psi(\tau_\ell)\mtx{1}{0}{y/p}{1}.
   \end{equation}
   Note that since $p^2|x$ and $p||(\ell-1)$, it follows that $(x+1-\ell)/p$ is not divisible by $p$. On the other hand, since $p||y$, we have that $y/p$ is not divisible by $p$.
   \begin{lemma}\label{lemma type III local}
    Suppose that $\ell$ is a trivial prime and that $\rho_{|\ell}$ satisfies $\cC_\ell$ of Type \rm{III}. Then, we have that $\delta_\ell(A')=0$. Moreover, if $\beta_{|\ell}=1$ it follows that $\delta_\ell(A)=1$ and $H^0(\Q_\ell, A)$ is finite.
\end{lemma}
\begin{proof}
  By \eqref{rho times psi eq 4} the action of $\tau_\ell$ on $A'$ is by $\mtx{1}{0}{y/p}{1}$. The same argument as in Lemma \ref{lemma type II local} then shows that $(A')^{\op{I}_\ell}$ is divisible and hence, $\mathcal{A}'=0$. It thus follows that $\delta_\ell(A')=0$. Now assume that $\beta_{|\ell}=1$. As mentioned above, since $\psi=\widetilde{\beta}^{-1}$, we have that $\psi(\sigma_\ell)=1$ and $\psi(\tau_\ell)=1$. As in the proof of Lemma \ref{lemma type II local}, 
\begin{equation}\label{AIeqn2}A^{\op{I}_\ell}=\left(p^{-1}\Z_p/\Z_p\right)e_1\oplus \Q_p/\Z_p e_2,\end{equation}
$(A^{\op{I}_\ell})_{\op{div}}=\Q_p/\Z_p e_2$ and $\mathcal{A}=\left(p^{-1}\Z_p/\Z_p\right)e_1$. We thus write 
    \[(\sigma_\ell-1)(p^{-1} e_1)=\left(\frac{x+1-\ell}{p}\right)e_2, \] and find that $(\sigma_\ell-1)(p^{-1} e_1)\in \left(\Z_p/\Z_p\right)e_1\oplus \left(\Z_p/\Z_p\right) e_2$. We have that
     $(\sigma_\ell-1) \mathcal{A}=0$ and therefore, $\delta_\ell(A)=1$. Next, we show that $H^0(\Q_\ell, A)$ is finite. Note that $H^0(\Q_\ell, A)=H^0(\langle \sigma_\ell \rangle, A^{\op{I}_\ell})$ and from \eqref{AIeqn2}, $A^{\op{I}_\ell}=\left(p^{-1}\Z_p/\Z_p \right)e_1 \oplus \Q_p/\Z_p e_2$. Thus it suffices to show that $\left(\Q_p/\Z_p e_2\right)^{\langle\sigma_\ell\rangle}$ is finite. By \eqref{rho times psi eq 3}, $\sigma_\ell$ acts on $A$ via the matrix $\mtx{1}{0}{(x+1-\ell)}{\ell}$ and thus 
     \[(\sigma_\ell-1)e_2=(\ell-1) e_2.\] Note that $p||(\ell-1)$ and therefore, $\left(\Q_p/\Z_p e_2\right)^{\langle\sigma_\ell\rangle}=\left(p^{-1}\Z_p/\Z_p\right)e_2$. Thus we have shown that $H^0(\Q_\ell, A)$ is finite.
\end{proof}

\begin{lemma}\label{Nl lemma}
    Suppose that $\ell$ is a trivial prime of Type \rm{III} and let $g\in \cN_\ell$. Then, there exist $a,b,c,d\in \F_q$ such that 
    \[g(\sigma_\ell)=\mtx{-a}{a}{b}{a}\text{ and } g(\tau_\ell)=\mtx{c}{0}{d}{-c}.\]
\end{lemma}
\begin{proof}
    Recall that $\mathcal{P}_\ell^{\op{ram}}$ is the space spanned by $\{f_1, f_2, g^{\op{ram}}\}$ and that $\cN_\ell$ is the conjugate of $\mathcal{P}_\ell^{\op{ram}}$ by $\mtx{0}{1}{1}{1}$. Let $f\in \mathcal{P}_\ell^{\op{ram}}$ be an element such that $g=\mtx{0}{1}{1}{1}f \mtx{0}{1}{1}{1}^{-1}$, and write 
    \[f=u f_1+v f_2+w g^{\op{ram}},\] where $u, v, w\in \F_q$. Set $y':=\frac{y}{\ell-1}$ and recall that 
    \[
\begin{split}
f_1(\sigma_\ell)=\mtx{0}{1}{0}{0} \text{, } & f_1(\tau_\ell)=\mtx{0}{0}{0}{0},\\ 
f_2(\sigma_\ell)=\mtx{0}{0}{0}{0} \text{, } & f_2(\tau_\ell)=\mtx{0}{1}{0}{0},\\  g^{\operatorname{ram}}(\sigma_\ell)=\mtx{0}{0}{1}{0} \text{, } & g^{\operatorname{ram}}(\tau_\ell)=\mtx{-y'}{0}{0}{y'}.
\end{split}\]
We find that 
\[\begin{split}& g(\sigma_\ell)=\mtx{0}{1}{1}{1} \mtx{0}{u}{w}{0}\mtx{-1}{1}{1}{0}=\mtx{-w}{w}{-w+u}{w}; \\ 
& g(\tau_\ell)=\mtx{0}{1}{1}{1} \mtx{-wy'}{v}{0}{wy'}\mtx{-1}{1}{1}{0}=\mtx{wy'}{0}{v+2wy'}{-wy'}.
\end{split}\] Setting $a:=w$, $b:=-w+u$, $c:=wy'$ and $d:=v+2wy'$, the result follows.
\end{proof}

Let $U_d$ (resp. $V_d$) be the $d$-dimensional $\F_p[\op{G}_{\Q}]$-submodules of $\op{Ad}^0\bar{\rho}$ (resp. $\op{Ad}^0\bar{\rho}^*$). Let $\mathcal{M}=\{U_1^*, U_2^*, U_3^*, V_1, V_2, V_3\}$, by \cite[Proposition 13]{hamblenramakrishna}, there is a finite set of primes $T\supset \Sigma$ such that $T\setminus \Sigma$ consists of trivial primes and $\Sh^1_T(M)=0$ for all $M\in \mathcal{M}$. At each of the primes $\ell\in T\setminus \Sigma$ the pair $(\cC_\ell, \mathcal{N}_\ell)$ is of Type \rm{I} (cf. \cite[p. 930]{hamblenramakrishna}). 

\begin{proposition}\label{type 1,2,3 propn}
    Suppose that $\rho:\op{G}_X\rightarrow \op{GL}_2(\Z_p)$ is a modular deformation of $\bar{\rho}$, and $X$ is a finite set of primes containing $T$, such that $X\setminus T$ consists of trivial primes. Moreover suppose that for integers $n\geq 1$ and $m\geq 0$, $X$ contains a set of primes $\{\ell_1, \dots, \ell_n, q_1 \dots, q_m\}$ disjoint from $T$ such that:
    \begin{itemize}
        \item $\rho_{|\ell_i}$ satisfies $\mathcal{C}_{\ell_i}$ of Type {\rm{III}} and $\beta_{|\ell_i}=1$ for $1\leq i\leq n$,
        \item $\rho_{|q_j}$ satisfies $\cC_{q_j}$ of Type {\rm{I}} for $1\leq j\leq m$,
        \item $\rho_{|\ell}$ satisfies $\mathcal{C}_\ell$ of Type {\rm{II}} or Type {\rm{III}} for $\ell\in X\setminus \left(T \cup \{\ell_1, \dots, \ell_n, q_1, \dots, q_m\}\right)$.
    \end{itemize}
    Then, we find that 
    \begin{equation}\label{dimension lower bound in terms of m and n}\dim \op{Sel}_{\op{BK}} (A'/\Q)[p]\geq n/2- 2 (m+\# T)-4.\end{equation}
\end{proposition}
\begin{proof}
Let us begin by checking the conditions of Proposition \ref{prop BK lower bound}. Note that the action on $A$ (resp. $A'$) is given by $\rho\otimes \psi$ (resp. $(\rho\otimes \psi)'$). The residual representations are as follows:
\[\overline{(\rho\otimes \psi)}=\mtx{\varphi \bar{\psi}}{\ast}{0}{\bar{\psi}}\text{ and }\overline{(\rho\otimes \psi)'}=\mtx{\varphi \bar{\psi}}{0}{\ast}{\bar{\psi}}.\]
Since $\bar{\chi}$ is odd, we find that $\varphi=\bar{\chi}\beta^2$ is odd. As $\beta$ is assumed to be unramified at $p$, it follows that $\varphi$ is ramified at $p$. The character $\bar{\psi}=\beta^{-1}$ is odd and unramified at $p$. The assumptions of Proposition \ref{mu=0 propn} are thus satisfied and it follows that $\mu(A)=0$. Since $\varphi \bar{\psi}$ is ramified at $p$ and $\bar{\psi}$ is odd, these characters are non-trivial, and hence $H^0(\Q, A)=0$ and $H^0(\Q, A')=0$. It follows from Lemma \ref{lemma type III local} that for $i=1, \dots, n$, we have that $H^0(\Q_{\ell_i}, A)$ is finite. By Proposition \ref{prop BK lower bound}, 
\[\op{dim}_{\F_p} \op{Sel}_{\op{BK}}(A'/\Q)[p]\geq n/2-\sum_{\ell\in X} \delta_\ell(A').\]
The primes $\ell\in X\setminus \left(T \cup \{q_1, \dots, q_m\}\right)$ are trivial primes with $\mathcal{C}_\ell$ of Type \rm{II} or \rm{III}. It follows from Lemmas \ref{lemma type II local} and \ref{lemma type III local} that $\delta_\ell(A')=0$. Therefore we have that 
\[\op{dim}_{\F_p} \op{Sel}_{\op{BK}}(A'/\Q)[p]\geq n/2-\sum_{\ell\in T\cup \{q_1, \dots, q_m\}} \delta_\ell(A').\]
It follows from Lemma \ref{explicit formula for tamagawa} that $\delta_\ell(A')\leq 2$ if $\ell\neq p$ and $\delta_p(A')\leq 6$. The result follows.
\end{proof}

Let $n$ be any large positive integer. We construct a lift $\rho: \op{G}_X \rightarrow \op{GL}_2(\Z_p)$ of $\bar{\rho}$ satisfying local conditions $\cC_\ell$ at each prime $\ell\in X$, satisfying the conditions of Proposition \ref{type 1,2,3 propn}. This is done so that $m\leq C$, where $C=C(\bar{\rho})$ is a positive constant which depends only on $\bar{\rho}$. Note that the set of primes $T$ depends only on $\bar{\rho}$. It then follows from \eqref{dimension lower bound in terms of m and n} that the dimension of $\op{Sel}_{\op{BK}} (A'/\Q)[p]$ is large.

\begin{proposition}\label{prop bounding dim of Selmer}
    For any $n>0$, there are $n$ trivial primes $Y_n=\{\ell_1, \dots, \ell_n\}$ disjoint from $T$ with $(\cC_{\ell_i}, \cN_{\ell_i})$ of Type \rm{III}, such that $\beta_{|\ell_i}=1$ and
    \[ \dim H^1_{\cN}(\op{G}_{T\cup Y_n}, \op{Ad}^0\bar{\rho})\leq \op{max}\{\dim H^1_{\cN}(\op{G}_{T}, \op{Ad}^0\bar{\rho}), 3\}.\]
\end{proposition}
\begin{proof}
The proof of the result is similar to that of \cite[Proposition 46]{hamblenramakrishna}. We provide a sketch of the details here. We inductively choose trivial primes $\ell_1, \dots, \ell_n$ such that $\beta_{|\ell_i}=1$ and 
    \[\dim H^1_{\cN}(\op{G}_{T\cup Y_{i}\cup\{\ell_{i+1}\}}, \op{Ad}^0\bar{\rho})\leq \op{max}\{\dim H^1_{\cN}(\op{G}_{T\cup Y_{i}}, \op{Ad}^0\bar{\rho}), 3\},\] where $Y_i=\{\ell_1, \dots, \ell_i\}$. Set $Y:=T\cup Y_i$, we show that there are infinitely many trivial primes $\ell$ such that $\beta_{|\ell}=1$ and with $(\cC_\ell, \cN_\ell)$ of Type \rm{III} for which 
    \[\dim H^1_{\cN}(\op{G}_{Y\cup\{\ell\}}, \op{Ad}^0\bar{\rho})\leq \op{max}\{\dim H^1_{\cN}(\op{G}_{Y}, \op{Ad}^0\bar{\rho}), 3\}.\]
First, consider the case when $\dim H^1_{\cN}(\op{G}_{Y}, \op{Ad}^0\bar{\rho})=0$. Note that the kernel of the map 
\begin{equation}\label{kernel of map}H^1_{\cN}(\op{G}_{Y\cup\{\ell\}}, \op{Ad}^0\bar{\rho})\rightarrow \frac{H^1(\Q_\ell, \op{Ad}^0\bar{\rho})}{H^1_{\op{nr}}(\Q_\ell, \op{Ad}^0\bar{\rho})}\end{equation} is contained in $H^1_{\cN}(\op{G}_{Y}, \op{Ad}^0\bar{\rho})$. In greater detail, if 
\[\alpha\in H^1_{\cN}(\op{G}_{Y\cup\{\ell\}}, \op{Ad}^0\bar{\rho})\] is in the kernel of \eqref{kernel of map} then it must be unramified at $\ell$; therefore, $\alpha\in H^1(\op{G}_Y, \op{Ad}^0\bar{\rho})$ via inflation. On the other hand, $\alpha$ satisfies all the Selmer conditions $\cN_q$ for all $q\in Y$, it follows that $\alpha\in H^1_{\cN}(\op{G}_Y, \op{Ad}^0\bar{\rho})$. Therefore $\alpha=0$ and the map \eqref{kernel of map} is injective. For any trivial prime $\ell$, one checks that
\[\op{dim}\left(\frac{H^1(\Q_\ell, \op{Ad}^0\bar{\rho})}{H^1_{\op{nr}}(\Q_\ell, \op{Ad}^0\bar{\rho})}\right)=3,\]and thus we find that 
\[\dim H^1_{\cN}(\op{G}_{Y\cup\{\ell\}}, \op{Ad}^0\bar{\rho})\leq 3.\]
\par Thus, assume that $\dim H^1_{\cN}(\op{G}_{Y}, \op{Ad}^0\bar{\rho})>0$. We show in this case that 
\[\dim H^1_{\cN}(\op{G}_{Y\cup\{\ell\}}, \op{Ad}^0\bar{\rho})\leq \dim H^1_{\cN}(\op{G}_{Y}, \op{Ad}^0\bar{\rho})\] for infinitely many trivial primes $\ell$ for which $\beta_{|\ell}=1$ and with $(\cC_\ell, \cN_\ell)$ of Type \rm{III}. 

\par It follows from \eqref{wiles formula}, \eqref{equation for dimension of Nell} and \cite[p. 938, l. 18]{hamblenramakrishna} that
\begin{equation}\label{selmer equal dselmer eqn}\dim H^1_{\cN^\perp}(\op{G}_{Y}, \op{Ad}^0\bar{\rho}^*) =\dim H^1_{\cN}(\op{G}_{Y}, \op{Ad}^0\bar{\rho}).\end{equation} Thus, there are non-zero elements $\psi\in H^1_{\cN^\perp}(\op{G}_{Y}, \op{Ad}^0\bar{\rho}^*)$ and $f\in \dim H^1_{\cN}(\op{G}_{Y}, \op{Ad}^0\bar{\rho})$. 
    
According to Proposition \ref{inf many triv primes f psi conditions}, there are infinitely many trivial primes $\ell$ for which 
    \begin{enumerate}
    \item $\beta_{|\ell}=1$,
    \item $f(\sigma_\ell)\neq \mtx{-a}{a}{b}{a}$ for any values of $a,b\in \F_q$, 
    \item $\gamma_{|\ell}=0$ for all $\gamma\in H^1(\op{G}_Y, U_1^*)$,
    \item $\psi_{|\ell}\neq 0$ and $\psi$ can be extended to a basis of $H^1_{\cN^\perp}(\op{G}_Y, \op{Ad}^0\bar{\rho}^*)$ such that $\alpha_{|\ell}=0$ for all other elements $\alpha$ of the basis.
\end{enumerate}
    We pick one such prime $\ell\notin Y$ and set $(\cC_\ell, \cN_\ell)$ to be of Type \rm{III}. Since $\gamma_{|\ell}=0$ for all $\gamma\in H^1(\op{G}_{Y}, U_1^*)$, it follows from \cite[Fact 18 (1) and p. 941, ll. 2-7]{hamblenramakrishna} that the maps 
    \begin{equation}\label{Ud equation}\begin{split}
        & H^1(\op{G}_{Y}, U_1)\rightarrow \bigoplus_{w\in Y} H^1(\Q_w, U_1),\\
        & H^1(\op{G}_{Y\cup \{\ell\}}, U_1)\rightarrow \bigoplus_{w\in Y} H^1(\Q_w, U_1)
    \end{split}\end{equation}
    have the same image, and the dimension of the kernel of the second map is $1$ more than that of the first map.
    Thus, there is an element $h_1\in H^1(\op{G}_{Y\cup\{\ell\}}, \op{Ad}^0\bar{\rho})$ such that 
    \[h_1(\sigma_\ell)=\mtx{0}{a}{0}{0}\text{ and }h_1(\tau_\ell)=\mtx{0}{1}{0}{0},\] where $a\in \F_q$.
    Suppose that there is a linear combination $sh_1+tf$ contained in $\cN_\ell$. We will show that $s=t=0$. Note that $f$ is unramified at $\ell$ and therefore $sh_1+t f$ maps $\tau_\ell$ to $\mtx{0}{s}{0}{0}$. On the one hand, according to Lemma \ref{Nl lemma}, $\cN_\ell$ consists of classes which take $\tau_\ell$ to $\mtx{\ast}{0}{\ast}{\ast}$. Thus $s$ must be $0$. On the other hand, since $f(\sigma_\ell)\neq \mtx{-a}{a}{b}{a}$ for all $a,b$, it follows from Lemma \ref{Nl lemma} that no non-trivial linear combination of $h_1$ and $f$ lies in $\cN_\ell$. From the proof of \cite[Proposition 46 and Fact 18 (2)]{hamblenramakrishna}, the kernel of 
    \begin{equation}\label{first map}H^1(\op{G}_{Y\cup\{\ell\}}, \op{Ad}^0\bar{\rho})\xrightarrow{\pi_Y'} \bigoplus_{w\in Y} H^1(\Q_w,\op{Ad}^0\bar{\rho})/\cN_w \end{equation}
    has dimension $2$ more than the kernel of
    \begin{equation}\label{second map}H^1(\op{G}_{Y}, \op{Ad}^0\bar{\rho})\xrightarrow{\pi_Y}\bigoplus_{w\in Y} H^1(\Q_w,\op{Ad}^0\bar{\rho})/\cN_w .\end{equation} 
    The kernel of \eqref{second map} equals $H^1_{\cN}(\op{G}_{Y}, \op{Ad}^0\bar{\rho})$ and we have that \begin{equation}\label{equation on ker pi Y 1}\dim \ker \pi_Y'=\dim \ker \pi_Y +2= \dim H^1_{\cN}(\op{G}_Y, \op{Ad}^0\bar{\rho})+2.\end{equation}
    Note that the kernel of \eqref{first map} contains $H^1_{\cN}(\op{G}_{Y\cup\{\ell\}}, \op{Ad}^0\bar{\rho})$ and so there is an exact sequence 
\[0\rightarrow H^1_{\cN}(\op{G}_{Y\cup\{\ell\}}, \op{Ad}^0\bar{\rho})\rightarrow \op{ker} \pi_Y'\rightarrow \frac{H^1(\Q_\ell, \op{Ad}^0\bar{\rho})}{\cN_\ell}.\]
 Since there are two elements in the kernel of \eqref{first map} for which no non-trivial linear combination lies in $\cN_\ell$, it thus follows that
\begin{equation}\label{equation on ker pi Y 2}\dim \ker\pi_Y'\geq \dim H^1_{\cN}(\op{G}_{Y\cup\{\ell\}}, \op{Ad}^0\bar{\rho})+2.\end{equation}
From \eqref{equation on ker pi Y 1} and \eqref{equation on ker pi Y 2} we deduce that 
\[\dim H^1_{\cN}(\op{G}_{Y\cup\{\ell\}}, \op{Ad}^0\bar{\rho})\leq \dim H^1_{\cN}(\op{G}_Y, \op{Ad}^0\bar{\rho}),\]
and this concludes the proof.
\end{proof}
\begin{theorem}\label{main thm aux}
   Let $p\geq 5$ be a prime number, $N$ be a natural number and
   \[\bar{\rho}=\mtx{\varphi}{\ast}{}{1}:\op{G}_{\Q}\rightarrow \op{GL}_2(\F_p),\] such that
\begin{enumerate}
    \item $\varphi=\bar{\chi} \beta^2$ where $\beta:\op{G}_{\Q}\rightarrow \F_p^\times$ is an odd character which is unramified at $p$, 
    \item $\beta^2\neq 1$,
    \item $\bar{\rho}$ is indecomposable.
\end{enumerate}
   Then, there exists a Galois representation $\rho: \op{G}_{\Q}\rightarrow \op{Aut}_{\Z_p}(T)\xrightarrow{\sim}\op{GL}_2(\Z_p)$ lifting $\bar{\rho}$ arising from a $p$-ordinary Hecke eigen-cuspform of weight $2$ such that 
   \[\dim_{\F_p}\op{Sel}_{\op{BK}}(A'/\Q)[p]\geq N.\]
\end{theorem}

\begin{proof}
   Recall that $\Sigma$ consists of $\ell$ at which $\bar{\rho}$ is ramified and $\ell=p$. According to \cite[Proposition 13]{hamblenramakrishna}, there is a finite set of primes $T\supset \Sigma$ such that $T\setminus \Sigma$ consists of trivial primes and $\Sh^1_T(M)=0$ for all $M\in \mathcal{M}$. At each of the primes $\ell\in T\setminus \Sigma$ the pair $(\cC_\ell, \mathcal{N}_\ell)$ is set to be of Type \rm{I} (see \cite[p. 930]{hamblenramakrishna}). Let $n$ be a natural number. (It turns out that $n/2$ will be a rough lower bound for the size of the $p$-torsion in the Bloch-Kato Selmer group.) Setting 
    \begin{equation}\label{defn of m'}m':=\dim H^1_{\cN} (\op{G}_{T}, \op{Ad}^0\bar{\rho}),\end{equation}
    it follows from Proposition \ref{prop bounding dim of Selmer} that there are $n$ trivial primes $Y_n=\{\ell_1, \dots, \ell_n\}$ disjoint from $T$ and $(\cC_{\ell_i}, \cN_{\ell_i})$ of Type \rm{III}, such that $\beta_{|\ell_i}=1$, and
    \[ \dim H^1_{\cN}(\op{G}_{T\cup Y_n}, \op{Ad}^0\bar{\rho})\leq \op{max}\{m', 3\}.\]
    Now let $X_1:=Y_n\cup T$. The obstruction to lifting $\bar{\rho}$ to a reducible \[\xi_2:\op{G}_{T}\rightarrow \op{GL}_2(\Z/p^2)\] lies in $H^2(\op{G}_{T}, U_1)$. Since $\bar{\rho}$ satisfied liftable conditions $\cC_\ell$ at all the primes $\ell\in X_1$, it follows that $\bar{\rho}_{|\ell}$ is liftable for all $\ell\in X_1$, and hence this obstruction lies in $\Sh^2_{T}(U_1)$. Note that $\Sh^1_{T}(U_1^*)=0$ and this implies that $\Sh^2_{T}(U_1)=0$ since $\Sh^1_{T}(U_1^*)\simeq \Sh^2_{T}(U_1)^\vee$. Thus, $\bar{\rho}$ lifts to $\xi_2$. Now for each prime $\ell\in X_1$, choose a local class $z_\ell\in H^1(\Q_\ell, \op{Ad}^0\bar{\rho})$ such that 
    \[(\op{Id}+z_\ell p) \xi_{2|\ell} \in \cC_\ell.\] It follows from \cite[Theorem 41]{hamblenramakrishna} (which uses the method of \cite{KLR}) that there is a set $Z$ of at most $2$ trivial primes $\ell$ with $(\cC_\ell, \cN_\ell)$ of Type \rm{II} and 
    \[z\in H^1(\op{G}_{X_1\cup Z}, \op{Ad}^0\bar{\rho})\] such that
    \begin{itemize}
        \item $(\op{Id}+p z) \xi_2$ satisfies $\cC_\ell$ for $\ell\in Z$,
        \item $z$ restricts to $z_\ell$ for all $\ell\in X_1$.
    \end{itemize}
Thus by construction, the twist $(\op{Id}+p z) \xi_2$ satisfies $\cC_\ell$ for all primes $\ell\in Z\cup X_1$. Setting $X_2:=X_1\cup Z$, and $\rho_2:=(\op{Id}+p z) \xi_2: \op{G}_{X_2}\rightarrow \op{GL}_2(\Z/p^2)$, we find that $\rho_2$ satisfies liftable conditions at all primes $\ell\in X_2$. Since $\Sh^2_{X_2}(\op{Ad}^0\bar{\rho})=0$, it follows that $\rho_2$ lifts to $\xi_3:\op{G}_{X_2}\rightarrow \op{GL}_2(\Z/p^3)$. The kernel of
\[H^1_{\cN}(\op{G}_{X_2}, \op{Ad}^0\bar{\rho})\rightarrow \bigoplus_{\ell\in Z} \frac{H^1(\Q_\ell, \op{Ad}^0\bar{\rho})}{H^1_{\op{nr}}(\Q_\ell, \op{Ad}^0\bar{\rho})}\] is contained in $H^1_{\cN}(\op{G}_{X_1}, \op{Ad}^0\bar{\rho})$. For a trivial prime $\ell$, it is easy to see that 
\[\op{dim}\left(\frac{H^1(\Q_\ell, \op{Ad}^0\bar{\rho})}{H^1_{\op{nr}}(\Q_\ell, \op{Ad}^0\bar{\rho})}\right)=3,\]and thus we find that 
\[\dim H^1_{\cN}(\op{G}_{X_2}, \op{Ad}^0\bar{\rho})\leq 3\# Z+\dim H^1_{\cN}(\op{G}_{X_1}, \op{Ad}^0\bar{\rho})\leq \op{max}\{m', 3\}+6.\]
It then follows from \cite[Proposition 46]{hamblenramakrishna} and Proposition \ref{inf many triv primes f psi conditions} that there are $m\leq \op{max}\{m',3\}+6$ trivial primes $q_1, \dots, q_m$, such that 
    \[H^1_{\cN}(\op{G}_X, \op{Ad}^0\bar{\rho})=0,\] where $X:=X_2\cup\{q_1, \dots, q_m\}$. Here, $(\cC_{q_i}, \cN_{q_i})$ is chosen of Type \rm{I} for $i=1, \dots, m$ and $\rho_{2|q_i}$ satisfies $\cC_{q_i}$ for $i=1, \dots, m$. We remark that although the primes $q_1, \dots, q_m$ depend on $\xi_3$, their cardinality $m$ is bounded in terms of $m'$ and hence depends only on $\bar{\rho}$. Recall from \eqref{selmer equal dselmer eqn} that
\[\dim H^1_{\cN^\perp}(\op{G}_{X}, \op{Ad}^0\bar{\rho}^*) =\dim H^1_{\cN}(\op{G}_{X}, \op{Ad}^0\bar{\rho})=0.\]

The vanishing of these Selmer groups and \eqref{PT les} implies that there is an isomorphism 
\begin{equation}\label{global selmer isomorphism}H^1(\op{G}_X, \op{Ad}^0\bar{\rho})\xrightarrow{\sim} \bigoplus_{\ell\in X} \left(\frac{H^1(\Q_\ell, \op{Ad}^0\bar{\rho})}{\mathcal{N}_\ell}\right).\end{equation}
For each prime $\ell \in X$, choose $f_\ell \in H^1(\Q_\ell, \operatorname{Ad}^0 \bar{\rho})$ such that 
\[
(\operatorname{Id} + p^2 f_\ell) \xi_{3|\ell} \in \mathcal{C}_\ell.
\]
Since $\mathcal{N}_\ell$ preserves $\mathcal{C}_\ell(\Z/p^3)$, any further translate of $f_\ell$ by an element of $\mathcal{N}_\ell$ also satisfies this condition. Thus, $f_\ell$ can be viewed as an element in the quotient $\frac{H^1(\Q_\ell, \operatorname{Ad}^0 \bar{\rho})}{\mathcal{N}_\ell}$.

The tuple $(f_\ell) \in \bigoplus_{\ell \in X} \frac{H^1(\Q_\ell, \operatorname{Ad}^0 \bar{\rho})}{\mathcal{N}_\ell}$ corresponds to a global cohomology class $f \in H^1(\operatorname{G}_X, \operatorname{Ad}^0 \bar{\rho})$ because the map \eqref{global selmer isomorphism} is an isomorphism. We set $\rho_3 = (\operatorname{Id} + p^2 f)\xi_3$, noting that $\rho_3$ satisfies the liftable deformation conditions $\mathcal{C}_\ell$ for all $\ell \in X$. The obstruction to lifting $\rho_3$ to $\xi_4: \operatorname{G}_X \rightarrow \operatorname{GL}_2(\Z/p^4)$ lies in $\Sh_X^2(\operatorname{Ad}^0 \bar{\rho})$, which is zero since $X$ contains $T$. Thus, a lift $\xi_4: \operatorname{G}_X \rightarrow \operatorname{GL}_2(\Z/p^4)$ exists. There is a tuple $(g_\ell) \in \bigoplus_{\ell \in X} \frac{H^1(\Q_\ell, \operatorname{Ad}^0 \bar{\rho})}{\mathcal{N}_\ell}$ defined by the property that for all $\ell \in X$,
\[
(\operatorname{Id} + p^3 g_\ell) \xi_{4|\ell} \in \mathcal{C}_\ell.
\]
This tuple corresponds to a global cohomology class $g \in H^1(\operatorname{G}_X, \operatorname{Ad}^0 \bar{\rho})$ because the map \eqref{global selmer isomorphism} is an isomorphism. Setting $\rho_4 = (\operatorname{Id} + p^3 g)\xi_4$, we find that $\rho_4$ satisfies the conditions $\mathcal{C}_\ell$ for all $\ell \in X$.

Continuing this process inductively, we obtain a compatible family of deformations $\rho_n: \operatorname{G}_X \rightarrow \operatorname{GL}_2(\Z/p^n)$ that satisfy the conditions $\mathcal{C}_\ell$ for all $\ell \in X$. Let $\rho: \operatorname{G}_X \rightarrow \operatorname{GL}_2(\Z_p)$ denote the inverse limit $\varprojlim_n \rho_n$. We find that $\rho$ is odd and satisfies the conditions $\mathcal{C}_\ell$ at the primes $\ell \in X$, including $\mathcal{C}_p$, which ensures it is ordinary at $p$. By Lemma \ref{new lemma 1}, $\rho_2$ is irreducible and thus $\rho$ is irreducible. By the result of Skinner and Wiles \cite[Theorem on p. 6]{skinnerwiles}, $\rho$ is modular and arises from a lattice $T_f$. Recall that the determinant $\nu=\det\rho$ is given by $\nu=\chi \widetilde{\beta}^2$. The weight is $k=2$, and that $\psi:=\widetilde{\beta}^{-1}$ satisfies $\psi^2=\chi\nu^{-1}$. 
We set $T:=T_f(\psi)$ to be the self dual twist of $T_f$ (not to be confused with the set of primes $T$ above), and let $T'$ and $A'$ be as at the beginning of Section \ref{s 5}.

    \par From Proposition \ref{type 1,2,3 propn}, we find that
     \[\dim \op{Sel}_{\op{BK}} (A'/\Q)[p]\geq n/2- 2 (m+\# T)-4\geq n/2- 2 (\op{max}\{m',3\}+8+\# T).\] Note that $m'$ is a constant which depends only on $\bar{\rho}$ on not on $\rho$ (see \eqref{defn of m'}). Also, the set $T$ only depends on $\bar{\rho}$. On the other hand, $n$ can be chosen large enough so that the right side of the above inequality is $\geq N$. This proves our result. 
\end{proof}

\begin{proof}[Proof of Theorem \ref{main thm}]
Given $p\geq 5$, it suffices to show that a representation $\bar{\rho}$ satisfying the conditions of Theorem \ref{main thm aux} does always exist. By Dirichlet's theorem on arithmetic progressions, there is a prime $\ell$ such that $\ell-1$ is divisible by $p-1$. Let $\chi_\ell$ be the mod $\ell$ cyclotomic character and set $\beta':=\chi_\ell^{\frac{\ell-1}{p-1}}$. Note that $\beta'$ has order $p-1$ and thus fixing some obvious identifications, gives a surjective character $\beta':\op{G}_\Q\rightarrow \F_p^\times$. Since $p\geq 5$, $\beta'$ is not a quadratic character and thus $(\beta')^2\neq 1$. Moreover, $\beta'$ is unramified at $p$ since $\chi_\ell$ is. If $\beta'$ is odd, then we set $\beta:=\beta'$. On the other hand, if $\beta'$ is even, set $\beta:=\beta' \alpha$ where $\alpha$ is an odd quadratic character which is unramified at $p$. Having chosen $\beta$, choose $\ast$ to be a non-zero cohomology class in the infinite-dimensional vector space $ H^1(\op{G}_{\Q}, \F_p(\varphi))$, for $\varphi:=\bar{\chi}\beta^2$. The representation $\bar{\rho}$ satisfies the conditions of Theorem \ref{main thm aux}, from which the result follows. 
\end{proof}
\bibliographystyle{alpha}
\bibliography{references}

\begin{thebibliography}{BKLOS21}

\bibitem[Bar10]{AlexBartel}
Alex Bartel.
\newblock Large {S}elmer groups over number fields.
\newblock {\em Math. Proc. Cambridge Philos. Soc.}, 148(1):73--86, 2010.

\bibitem[BK90]{BlochKato}
Spencer Bloch and Kazuya Kato.
\newblock {$L$}-functions and {T}amagawa numbers of motives.
\newblock In {\em The {G}rothendieck {F}estschrift, {V}ol. {I}}, volume~86 of {\em Progr. Math.}, pages 333--400. Birkh\"{a}user Boston, Boston, MA, 1990.

\bibitem[BKLOS21]{BKLS}
Manjul Bhargava, Zev Klagsbrun, Robert~J. Lemke~Oliver, and Ari Shnidman.
\newblock Elements of given order in {T}ate-{S}hafarevich groups of abelian varieties in quadratic twist families.
\newblock {\em Algebra Number Theory}, 15(3):627--655, 2021.

\bibitem[{\v{C}es}na17]{cesnavicius}
Kestutis {\v{C}es}navi\v{c}ius.
\newblock {$p$}-{S}elmer growth in extensions of degree {$p$}.
\newblock {\em J. Lond. Math. Soc. (2)}, 95(3):833--852, 2017.

\bibitem[FKP22]{FKP}
Najmuddin Fakhruddin, Chandrashekhar Khare, and Stefan Patrikis.
\newblock Lifting and automorphy of reducible {${\rm mod}\,p$} {G}alois representations over global fields.
\newblock {\em Invent. Math.}, 228(1):415--492, 2022.

\bibitem[Fla90]{Flach}
Matthias Flach.
\newblock A generalisation of the {C}assels-{T}ate pairing.
\newblock {\em J. Reine Angew. Math.}, 412:113--127, 1990.

\bibitem[Fly18]{Flynn}
E.~V. Flynn.
\newblock Arbitrarily large {T}ate-{S}hafarevich group on {A}belian surfaces.
\newblock {\em J. Number Theory}, 186:248--258, 2018.

\bibitem[FM95]{fontainemazur}
Jean-Marc Fontaine and Barry Mazur.
\newblock Geometric {G}alois representations.
\newblock In {\em Elliptic curves, modular forms, \& {F}ermat's last theorem ({H}ong {K}ong, 1993)}, Ser. Number Theory, I, pages 41--78. Int. Press, Cambridge, MA, 1995.

\bibitem[FW79]{FerreroWash}
Bruce Ferrero and Lawrence~C. Washington.
\newblock The {I}wasawa invariant {$\mu _{p}$} vanishes for abelian number fields.
\newblock {\em Ann. of Math. (2)}, 109(2):377--395, 1979.

\bibitem[Gre89]{Greenbergpadic}
Ralph Greenberg.
\newblock Iwasawa theory for {$p$}-adic representations.
\newblock In {\em Algebraic number theory}, volume~17 of {\em Adv. Stud. Pure Math.}, pages 97--137. Academic Press, Boston, MA, 1989.

\bibitem[Gre97]{Greenbergstructure}
Ralph Greenberg.
\newblock The structure of {S}elmer groups.
\newblock volume~94, pages 11125--11128. 1997.
\newblock Elliptic curves and modular forms (Washington, DC, 1996).

\bibitem[Gre99]{GreenbergCetaro}
Ralph Greenberg.
\newblock Iwasawa theory for elliptic curves.
\newblock In {\em Arithmetic theory of elliptic curves ({C}etraro, 1997)}, volume 1716 of {\em Lecture Notes in Math.}, pages 51--144. Springer, Berlin, 1999.

\bibitem[HR08]{hamblenramakrishna}
Spencer Hamblen and Ravi Ramakrishna.
\newblock Deformations of certain reducible {G}alois representations. {II}.
\newblock {\em Amer. J. Math.}, 130(4):913--944, 2008.

\bibitem[Kat04]{Katopadichodge}
Kazuya Kato.
\newblock {$p$}-adic {H}odge theory and values of zeta functions of modular forms.
\newblock Number 295, pages ix, 117--290. 2004.
\newblock Cohomologies $p$-adiques et applications arithm\'{e}tiques. III.

\bibitem[KLR05]{KLR}
Chandrashekhar Khare, Michael Larsen, and Ravi Ramakrishna.
\newblock Constructing semisimple {$p$}-adic {G}alois representations with prescribed properties.
\newblock {\em Amer. J. Math.}, 127(4):709--734, 2005.

\bibitem[LV21]{LongoVigni}
Matteo Longo and Stefano Vigni.
\newblock On {B}loch-{K}ato {S}elmer groups and {I}wasawa theory of {$p$}-adic {G}alois representations.
\newblock {\em New York J. Math.}, 27:437--467, 2021.

\bibitem[Mat07]{matsuno}
Kazuo Matsuno.
\newblock Construction of elliptic curves with large {I}wasawa {$\lambda$}-invariants and large {T}ate-{S}hafarevich groups.
\newblock {\em Manuscripta Math.}, 122(3):289--304, 2007.

\bibitem[NSW08]{NSW}
J\"{u}rgen Neukirch, Alexander Schmidt, and Kay Wingberg.
\newblock {\em Cohomology of number fields}, volume 323 of {\em Grundlehren der mathematischen Wissenschaften [Fundamental Principles of Mathematical Sciences]}.
\newblock Springer-Verlag, Berlin, second edition, 2008.

\bibitem[Och00]{Ochiai}
Tadashi Ochiai.
\newblock Control theorem for {B}loch-{K}ato's {S}elmer groups of {$p$}-adic representations.
\newblock {\em J. Number Theory}, 82(1):69--90, 2000.

\bibitem[Pan22]{Pan}
Lue Pan.
\newblock The {F}ontaine-{M}azur conjecture in the residually reducible case.
\newblock {\em J. Amer. Math. Soc.}, 35(4):1031--1169, 2022.

\bibitem[Pat16]{PatrikisExceptional}
Stefan Patrikis.
\newblock Deformations of {G}alois representations and exceptional monodromy.
\newblock {\em Invent. Math.}, 205(2):269--336, 2016.

\bibitem[Ram99]{RaviInventMath}
Ravi Ramakrishna.
\newblock Lifting {G}alois representations.
\newblock {\em Invent. Math.}, 138(3):537--562, 1999.

\bibitem[Ram02a]{RamakrishnaJRMS}
Ravi Ramakrishna.
\newblock Deformations of certain reducible {G}alois representations.
\newblock {\em J. Ramanujan Math. Soc.}, 17(1):51--63, 2002.

\bibitem[Ram02b]{RamakrishnaFM}
Ravi Ramakrishna.
\newblock Deforming {G}alois representations and the conjectures of {S}erre and {F}ontaine-{M}azur.
\newblock {\em Ann. of Math. (2)}, 156(1):115--154, 2002.

\bibitem[Ray21]{Raydefreducible}
Anwesh Ray.
\newblock Deformations of certain reducible {G}alois representations {III}.
\newblock {\em Int. J. Number Theory}, 17(6):1429--1485, 2021.

\bibitem[Rub00]{RubinES}
Karl Rubin.
\newblock {\em Euler systems}, volume 147 of {\em Annals of Mathematics Studies}.
\newblock Princeton University Press, Princeton, NJ, 2000.
\newblock Hermann Weyl Lectures. The Institute for Advanced Study.

\bibitem[SW99]{skinnerwiles}
C.~M. Skinner and A.~J. Wiles.
\newblock Residually reducible representations and modular forms.
\newblock {\em Inst. Hautes \'{E}tudes Sci. Publ. Math.}, (89):5--126 (2000), 1999.

\bibitem[Tay03]{tayloricosahedral}
Richard Taylor.
\newblock On icosahedral {A}rtin representations. {II}.
\newblock {\em Amer. J. Math.}, 125(3):549--566, 2003.

\end{thebibliography}
\end{document}